\documentclass[aAP,preprint]{imsart}

\RequirePackage[OT1]{fontenc}
\RequirePackage[numbers]{natbib}
\RequirePackage[colorlinks,citecolor=blue,urlcolor=blue]{hyperref}\usepackage[latin1]{inputenc}
\usepackage{amsmath,amsfonts,amssymb,amsthm,amsopn,amscd}
\usepackage{bbm,wasysym,stmaryrd}
\startlocaldefs
\newcommand{\R}{\mathbbm{R}}

\newcommand{\Z}{\mathbbm{Z}}

\newcommand{\F}{\mathcal{F}}

\newcommand{\Exp}{\mathbb{E}} 

\newcommand{\T}{\mathcal{T}}  

\newcommand{\one}{\mathfrak{1}}

\newcommand{\erf}{\rm{erf}}

\newcommand{\Prob}{\mathbb{P}}

\newcommand{\PP}{\mathbb{P}}

\newcommand{\C}{\mathcal{C}}

\newcommand{\norm}[1]{\left\| #1 \right\|}
\newcommand{\lsup}[1]{\underset{#1\to\infty}{\overline{\lim}}}
\newcommand{\linf}[1]{\underset{#1\to\infty}{\underline{\lim}}}

\newcommand{\modd}{\text{ mod }}

\theoremstyle{plain}
\newtheorem{theorem}{Theorem}

\newtheorem{proposition}[theorem]{Proposition}
\newtheorem{lemma}[theorem]{Lemma}

\newtheorem{assumption}{Assumption}
\newtheorem{Remark}[theorem]{Remark}

\endlocaldefs

\begin{document}
\begin{frontmatter}
\title{Large Deviations of a Spatially Stationary Network of Interacting Neurons}
\runtitle{Large Deviations of a Network of Interacting Particles}
\begin{aug}
\author{\fnms{Olivier} \snm{Faugeras}\ead[label=e1]{olivier.faugeras@inria.fr}} and 
\author{\fnms{James} \snm{MacLaurin}\corref{}\ead[label=e2]{j.maclaurin@sydney.edu.au}},
\runauthor{O. Faugeras et al.}

\affiliation{NeuroMathComp\thanksmark{m1} INRIA Sophia Antipolis}

\address{
NeuroMathComp INRIA \\
2004 Route Des Lucioles\\
B.P. 93, 06902, Sophia Antipolis France\\
\printead{e1,e2}
}
\end{aug}
 
 \begin{abstract}
In this work we determine a process-level Large Deviation Principle (LDP)  for a model of interacting neurons indexed by a lattice $\Z^d$. The neurons are subject to noise, which is modelled as a correlated martingale. The probability law governing the noise is strictly stationary, and we are therefore able to find a LDP for the probability laws $\Pi^n$ governing the stationary empirical measure $\hat{\mu}^n$ generated by the neurons in a cube of length $(2n+1)$. We use this LDP to determine an LDP for the neural network model. The connection weights between the neurons evolve according to a learning rule / neuronal plasticity, and these results are adaptable to a large variety of neural network models. This LDP is of great use in the mathematical modelling of neural networks, because it allows a quantification of the likelihood of the system deviating from its limit, and also a determination of which direction the system is likely to deviate. The work is also of interest because there are nontrivial correlations between the neurons even in the asymptotic limit, thereby presenting itself as a generalisation of traditional mean-field models.
\end{abstract}
 
 \begin{keyword}[class=MSC]
\kwd[Primary ]{60F10}
\kwd[; secondary ]{60H20,92B20,68T05,82C32}
\end{keyword}

\begin{keyword}
\kwd{Large Deviations}
\kwd{ergodic}
\kwd{neural network}
\kwd{learning}
\kwd{SDE}
\kwd{lattice}
\kwd{interacting particles}
\kwd{stationary}
\kwd{process level}
\kwd{level 3}
\kwd{empirical measure}
\kwd{periodic}
\end{keyword}
 
\end{frontmatter}
\section{Introduction}
In this paper we determine a Large Deviation Principle for a strictly stationary model of interacting processes on a lattice. We are motivated in particular by the study of interacting neurons in neuroscience, but this work ought to be adaptable to other phenomena such as mathematical finance, population genetics or insect swarms. In neuroscience, neurons form complicated networks which may be studied on many levels. On the macroscopic level, neural field equations model the density of activity per space / time. They have been very successful in understanding many phenomena in the brain, including visual hallucinations \cite{ermentrout-cowan:79,bressloff-cowan-etal:02}, motion perception \cite{geise:99}, feature selectivity in the visual cortex \cite{hansel-sompolinsky:98b} and traveling waves \cite{ermentrout-mcleod:93,tuckwell:08,pinto-ermentrout:01,kilpatrick-bressloff:10,faye:13,bressloff:14}. On the microscopic level, models such as that of Hodgkin and Huxley explain the dynamics of action-potentials very accurately. One of the most important outstanding questions in mathematical neuroscience is a detailed and mathematically rigorous derivation of the macroscopic from the microscopic equations \cite{bressloff:12,touboul:14}. In particular, perhaps two of the most difficult phenomena to model are the nature of the connection strengths between the neurons, and the stochastic noise. We will discuss these further below, but before we do this we provide a brief introduction to mean-field models of neuroscience.

Classical mean-field models are perhaps the most common method used to scale up from the level of individual neurons to the level of populations of neurons \cite{baladron-fasoli-etal:12b,touboul:14}. For a group of neurons indexed from $1$ to $N$, the evolution equation of a mean field model is typically of the following form (an $\R^N$-valued SDE)
\begin{equation}
dX^j_t = \left[g(X^j_t) + \frac{1}{N}\sum_{k=1}^N h_t(X^j,X^k)\right]dt + \sigma(X^j_t)dW^j_t.\label{eq:meanfieldbasic}
\end{equation}
We set $X^j_0= 0$. Here $g$ is Lipschitz, $h$ is Lipschitz and bounded, and $\sigma$ is Lipschitz. $(W^j)$ are independent Brownian Motions representing internal / external noise. Asymptoting $N$ to $\infty$, we find that in the limit $X^j$ is independent of $X^k$ (for $j\neq k$), and each $X^j$ is governed by the same law \cite{sznitman:91}. Since the $(X^j)$ become more and more independent, it is meaningful to talk of their mean as being representative of the group as a whole. In reaching this limit, three crucial assumptions have been made: that the external synaptic noise is uncorrelated, that the connections between the neurons are homogeneous and that the connections are scaled by the inverse of the size of the system. We will relax each of these assumptions in our model (which is outlined in Section \ref{sect:NeuronModel}).

The noise has a large effect on the limiting behavior, but as already noted it is not necessarily easy to model. Manwani and Koch \cite{manwani-koch:99} distinguish three main types of noise in the brain: thermal, channel noise and synaptic noise coming from other parts of the brain. With synaptic noise in particular, it is not clear to what extent this is indeed `noise', or whether there are correlations or neural coding that we are not yet aware of. At the very least, we expect that the correlation in the synaptic noise affecting two neurons close together should be higher than the correlation in the synaptic noise affecting two neurons a long way apart. The signal output of neurons has certainly been observed to be highly correlated \cite{sompolinsky-yoon-etal:01,schneidman-berry-etal:06,averbeck-latham-etal:06}. In our model for the synaptic noise in Section \ref{Section Noise LDP}, the noise is correlated, with the correlation determined by the lattice distance between the neurons. Indeed the probability law for the noise is stationary relative to the toroidal topology of our neural network, meaning that it is invariant under rotations of the torus. 

The other major difference between the model in Section \ref{sect:NeuronModel} and the mean field model outlined above is the model of the synaptic connections. In the study of emergent phenomena of interacting particles, the nature of the connections between the particles is often more important than the particular dynamics governing each individual \cite{haken:06}. One of the reasons the synaptic connections are scaled by the inverse of the number of neurons is to ensure that the mean-field equation \eqref{eq:meanfieldbasic} has a limit as $N\to\infty$. However this assumption, while useful, appears a little \textit{ad hoc}. One might expect that the strength of the synaptic connections is independent of the population size, and rather the system does not `blowup'  for large populations because the strength of the connections decays with increasing distance. This is certainly the standard assumption in models of the synaptic kernel in neural field models \cite{bressloff:12}. Furthermore the asymptotic behaviour in the mean-field model is extremely sensitive to the scaling. For example, if we were to scale the synaptic strength by $N^{-\beta}$ for any $\beta > 1$, then the limiting asymptotic behaviour would be quite different (the limiting law would not be McKean-Vlasov). Finally there is a lot of evidence that the strength of connection evolves in time through a learning rule / neural plasticity \cite{gerstner-kistler:02,galtier:11}. We will incorporate these effects into our model of the synaptic weights, and ensure in addition that they are such that the probability law is stationary relative to the toroidal topology. We note that there already exists a literature on the asymptotic analysis of interacting diffusions, including \cite{cox-fleischmann-etal:96,liggett:05,greven-hollander:07}. Most of this literature is concerned with the ergodic behaviour in time of a countably infinite set of interacting diffusions, whereas this paper is more focussed on understanding the behaviour over a fixed time interval of an asymptotically large network of interacting neurons.

The main result of this paper is a Large Deviation Principle (LDP) for the neural network model in Section \ref{sect:NeuronModel}. This essentially gives the exponential rate of convergence towards the limit. A Large Deviation Principle is a very useful mathematical technique which allows us to estimate finite-size deviations of the system from its limit behaviour. There has been much effort in recent years to understand such finite-size phenomena in mathematical models of neural networks - see for instance \cite{bressloff:09,buice-cowan-etal:10,touboul-ermentrout:11,fasoli:13,fasoli2015complexity}. More generally, there has already been considerable work in the Large Deviations of ergodic phenomena. Donsker and Varadhan obtained a Large Deviations estimate for the law governing the empirical process generated by a Markov Process \cite{donsker-varadhan:83}. They then determined a Large Deviations Principle for an (integer-indexed) stationary Gaussian Process, obtaining a particularly elegant expression for the rate function using spectral theory. \cite{chiyonobu-kusuoka:88,deuschel-stroock-etal:91,bryc-dembo:96} obtain a Large Deviations estimate for the empirical measure generated by processes satisfying various mixing conditions and \cite{georgii1993large} obtain an LDP for stationary Gibbs Measures. \cite{baxter-jain:93} obtain a variety of results for Large Deviations of ergodic phenomena, including one for the Large Deviations of $\Z$-indexed $\R^T$-valued stationary Gaussian processes. \cite{rassoul2011process,kubota2012large,fukushima2014quenched} obtain large deviation principles for a random walk in a random environment. There also exists a literature modelling the Large Deviations and other asymptotics of weakly-interacting particle systems (see for example \cite{dawson-gartner:87,ben-arous-guionnet:95,goldys:01,dawson-del-moral:05,budhiraja-dupuis-etal:12,fischer:12,giacomin-lucon-etal:11,lucon:12,lucon-stannat:13,faugeras-maclaurin:14d,faugeras-maclaurin:15,cabana2015large,bossy2015clarification}). These are systems of $N$ particles, each evolving stochastically, and usually only interacting via the empirical measure. 

The correlations in the noise together with the inhomogeneity of the synaptic weight model mean that the limit equation as $n\to \infty$ of \eqref{eq:fundamentalmult} is not asynchronous, unlike \eqref{eq:meanfieldbasic} (see \cite{ginzburg-sompolinsky:94} for a discussion of (a)synchronicity). Indeed the neurons are potentially highly correlated, even in the large system limit. This means that the results of this paper would be well-suited for further investigation of stochastic resonance \cite{brunel-hansel:06,ostojic-brunel-etal:09,mcdonnell-ward:11,giacomin-lucon-etal:11}. Furthermore, one could obtain an LDP for the asymptotics of the synaptic weight connections $\Lambda_s^k(U^j,U^{j+k})$ through an application of the contraction principle (see \cite[Theorem 4.2.1]{dembo-zeitouni:97}) to Theorem \ref{Theorem Main LDP}. This would be of interest in understanding the asymptotics of the network architecture in the large size limit.

This paper is structured as follows. In Section \ref{sect:stochprocess} we outline a general model of interacting neurons on a lattice, and state a large deviation principle under a set of assumptions. In Section \ref{Section Proofs} we prove this theorem. In Section \ref{Section Noise LDP} we outline a model of the noise as a correlated martingale, and prove a large deviation principle for the law of the empirical measure. In Section \ref{Section Example} we outline an extended example of this theory which satisfies the assumptions of Section \ref{sect:stochprocess}. This example considers a Fitzhugh-Nagumo model of interacting neurons, with Hebbian learning on the synaptic weights and subject to the correlated noise of Section \ref{Section Noise LDP}. 
\section{Outline of Model and Preliminary Definitions}\label{sect:stochprocess}
In this section we start by outlining our finite model of $(2n+1)^d$ stationary interacting neurons indexed over $V_n$. In Section \ref{Section Assumptions} we outline our assumptions on the model. The main result of this paper is in Theorem \ref{Theorem Main LDP}. 
\subsection{Preliminaries}\label{Subsection Preliminaries}
We must first make some preliminary definitions. Let $(\Omega,\F,\Prob)$ be a complete probability space. If $X$ is some separable topological space, then we denote the $\sigma$-algebra generated by the open sets by $\mathcal{B}(X)$, and the set of all probability measures on $(X,\mathcal{B}(X))$ by $\mathcal{P}(X)$. We endow $\mathcal{P}(X)$ with the topology of weak convergence. 

Elements of the processes in this paper are indexed by the lattice points $\Z^d$: for $j\in\Z^d$ we write $j = (j(1),\ldots,j(d))$. Let $V_n \subset \Z^d$ be such that $j\in V_n$ if $|j(m)| \leq n $ for all $1\leq m \leq d$. The number of elements in $V_n$ is written as $|V_n| := (2n+1)^d$.

We assume that the state space for each neuron is $\R$. For any $s\in [0,T]$, we endow  $\mathcal{C}([0,s],\R)$ with the norm $\norm{U}_s := \sup_{r\in [0,s]}|U_r|$.  Write $\mathcal{T} := \mathcal{C}\big([0,T],\R\big)$. Let $\lbrace\lambda^j\rbrace_{j\in\Z^d}$ be a set of weights, satisfying $\lambda^j > 0$ and $\sum_{j\in\Z^d}\lambda^j = 1$. We make further assumptions about $\lbrace \lambda^j\rbrace$ at the start of Section \ref{Section Proofs}. Let $\T^{\Z^d}_\lambda$ be the separable Banach Space of all $U := (U^j)_{j\in\Z^d} \in \T^{\Z^d}$ such that the following norm is finite
\begin{equation}
\norm{U}_{T,\lambda} := \sqrt{\sum_{j\in\Z^d}\lambda^j \norm{U^j}_T^2} < \infty.
\end{equation}
Let $\pi^{V_m}: \T^{\Z^d}\to \T^{V_m}$ be the projection $\pi^{V_m}(X) := (X^j)_{j\in V_m}$. It can be checked that the embedding $\T^{\Z^d}_\lambda \hookrightarrow \T^{\Z^d}$ is continuous when $\T^{\Z^d}$ is endowed with the cylindrical topology (generated by sets $O\subset \T^{\Z^d}$ such that $\pi^{V_m}O$ is open in $\T^{V_m}$). Let $d^{\lambda,\mathcal{P}}$ be the Levy-Prokhorov metric on $\mathcal{P}(\T_{\lambda}^{\Z^d})$ generated by the norm $\norm{\cdot}_{T,\lambda}$ on $\T^{\Z^d}_\lambda$.
\subsection{Outline of Model and Main Result}\label{sect:NeuronModel}

For $n\in\Z^+$, there are $|V_n|$ neurons in our network. There are three components to the dynamics of our neural network model: the internal dynamics term $\mathfrak{b}_s$, the interaction term $\Lambda^k_s(U^j,U^{(j+k)\modd V_n})$ and the noise term $W^{n,j}_t$. The form of our interaction term differs from standard mean-field models in that it is not scaled by some function of $|V_n|$, and it is not homogenous throughout the network. Rather the function itself depends on the lattice distance $k$ between the neurons (the distance being taken modulo $V_n$). We must make some assumptions on the behaviour of $\Lambda^k_t$ when $|k|$ is large to ensure that the system is convergent (see Assumption \ref{Assumption NonUniform}). The interaction $\Lambda^k_t$ can also be a function of the past activity, which allows us to incorporate both delays in the signal transmission and a learning model for the synaptic weights.

The form of the interaction explains why we work in the weighted space $\T^{\Z^d}_\lambda$ rather than $\T^{\Z^d}$. We will choose the weights $(\lambda^j)_{j\in\Z^d}$ carefully so that they dominate the interaction terms (in a certain sense). The mapping from $W^n \to U$ of the noise to the solution (which we define to be $\Psi^n$ in \eqref{eq:fundamentalmult 3}) will then be Lipschitz, uniformly in $n$, relative to the topology of $\T^{\Z^d}_\lambda$. If we were to work in the space $\T^{\Z^d}$ endowed with the cylindrical topology, then the mappings $\Psi^n$ and $\Psi$ ($\Psi$ is the limit  as $n\to \infty$ of $\Psi^n$) would not (in general) be continuous, even if the interactions were zero beyond some fixed lattice distance. We thus choose to work in $\T^{\Z^d}_\lambda$ because we can take advantage of the fact that LDPs are preserved under continuous maps . In any case the LDP over $\T^{\Z^d}_\lambda$ of Theorem \ref{Theorem Main LDP} is a stronger result; the LDP over $\T^{\Z^d}$ under the cylindrical topology is an immediate corollary of this, as we note in Remark \ref{Remark 1}. 
 
The system we study in this paper is governed by the following evolution equation: for $j\in V_n$,
\begin{equation}\label{eq:fundamentalmult}
U^{j}_t = U_{ini} + \int_0^t \bigg(\mathfrak{b}_s(U^{j}) + \sum_{k\in V_n}\Lambda^{k}_s(U^j,U^{(j+k)\modd V_n})\bigg)ds + W^{n,j}_{t}.
\end{equation}

Here $(j+k)\modd V_n := l\in V_n$, such that $(j(p)+k(p)) \modd (2n+1) = l(p)$ for all $1\leq p\leq d$. Thus one may think of the neurons as existing on a torus. This is what we meant when we stated in the introduction that the network has a `toroidal topology'. $U_{ini} \in \R$ is some constant. It follows from Lemmas \ref{lem: Psin Psi} and \ref{Lemma Solution Equivalence} further below that there exists a unique solution to \eqref{eq:fundamentalmult} $\PP$-almost surely. The thrust of this article is to understand the asymptotic behaviour of the network as $n\to \infty$.

We assume that $W^{n} := \big(W^{n,j}_{t}\big)_{j\in V_n,t\in [0,T]}$ is a $\mathcal{T}^{V_n}$-valued random variable such that $W^{n,j}_0 = 0$. In Section \ref{Section Noise LDP}, we outline an example model for the $(W^{n,j}_{t})$ where for any two $j,k\in V_n$, $W^{n,j}_t$ is correlated with $W^{n,k}_t$, and each $W^{n,j}_t$ is a martingale in time. However this model is not necessary for Theorem \ref{Theorem Main LDP} to be valid. It is important to be aware that for some fixed $m\in\Z^+$, the law of $(W^{n,j})_{j\in V_m}$ may vary with $n$ (which is the case for the example in Section \ref{Section Noise LDP}, where the correlations are modulo $V_n$). 

Let $S^k:\T^{\Z^d} \to \T^{\Z^d}$ (for some $k\in\Z^d$) be the shift operator (i.e. $(S^k x)^m := x^{m+k}$). Denote the empirical measure $\hat{\mu}^n: \T^{V_n} \to \mathcal{P}(\T_{\lambda}^{\Z^d})$ by
\begin{equation}
\label{defn empirical measure}
\hat{\mu}^n(X) := \frac{1}{|V_n|}\sum_{j\in V_n}\delta_{S^j \tilde{X}},
\end{equation}
where $\tilde{X}\in\T^{\Z^d}_\lambda$ is the $V_n$-periodic interpolant, i.e. $\tilde{X}^j := \tilde{X}^{j \modd V_n}$. If $X \in \T^{\Z^d}$, then in a slight abuse of notation we write $\hat{\mu}^n(X) := \hat{\mu}^n(\pi^{V_n}X)$. It may be noted that $\hat{\mu}^n$ is stationary, i.e. $\hat{\mu}^n(X)\circ (S^j)^{-1} = \hat{\mu}^n(X)$ for any $j\in\Z^d$.

We now outline our main result.
\begin{theorem}\label{Theorem Main LDP}
Let the law of $\hat{\mu}^n(U)$ be $\Pi^n \in \mathcal{P}(\mathcal{P}(\T_{\lambda}^{\Z^d}))$. Under the assumptions outlined in Section \ref{Section Assumptions}, $(\Pi^n)_{n\in\Z^+}$ satisfy a Large Deviation Principle with good rate function $I$ (i.e. $I$ has compact level sets). That is, for all closed subsets $A$ of $\mathcal{P}\big(\T_{\lambda}^{\Z^d}\big)$,
\begin{equation}\label{LDP closed}
\lsup{n}\frac{1}{|V_n|}\log\Pi^n(A) \leq - \inf_{\gamma\in A}I(\gamma).
\end{equation}
For all open subsets $O$ of $\mathcal{P}\big(\T_{\lambda}^{\Z^d}\big)$,
\begin{equation}
\linf{n}\frac{1}{|V_n|}\log\Pi^n(O) \geq - \inf_{\gamma\in O}I(\gamma).\label{LDP open}
\end{equation}
\end{theorem}
From now on, if a sequence of probability laws satisfies \eqref{LDP closed}  and \eqref{LDP open} for some $I$ with compact level sets, then to economise space we say that it satisfies an LDP with a good rate function.

\begin{Remark}
We make some brief comments on the form of the rate function $I$. It is infinite outside the set of all stationary measures: that is if $\mu \in \mathcal{P}(\T^{\Z^d}_\lambda)$ is such that $\mu \circ (S^j)^{-1} \neq \mu$ for some shift $S^j$, then $I(\mu) = \infty$. This is because $\hat{\mu}^n(U)$ is stationary and the set of all stationary measures is closed. In many circumstances one could use results concerning the specific relative entropy to obtain a convenient expression for the rate function $I$: see for example \cite{georgii1993large}. 
\end{Remark}

\begin{Remark}\label{Remark 1}
Because of the continuity of the embeddings $\T^{\Z^d}_\lambda \hookrightarrow \T^{\Z^d}$ and $\mathcal{P}(\T^{\Z^d}_\lambda) \hookrightarrow \mathcal{P}(\T^{\Z^d})$, we may  infer an LDP for $(\Pi^n)_{n\in\Z^+}$ relative to the weak topology on $\mathcal{P}(\T^{\Z^d})$ induced by the cylindrical topology on $\T^{\Z^d}$. This follows directly from an application of the Contraction Principle \cite[Theorem 4.2.21]{dembo-zeitouni:97} to Theorem \ref{Theorem Main LDP}.
\end{Remark} 

\begin{Remark}
It is worth noting that the modulo $V_n$ form of the interaction is not essential to get an LDP of this type. One could for example obtain a similar result to Theorem \ref{Theorem Main LDP} through replacing $\Lambda^k_s(U^j,U^{(j+k)\modd V_n})$ in \eqref{eq:fundamentalmult} with $0$ if $(j+k) \notin V_n$. We have chosen this form because we find it more elegant.
\end{Remark}
\subsection{Assumptions}\label{Section Assumptions}

We employ the following assumptions. 

Let the law of $\hat{\mu}^n(W^n)$ be $\Pi^n_W \in \mathcal{P}\big(\mathcal{P}(\T^{\Z^d}_\lambda)\big)$. We will obtain the LDP for $(\Pi^n)_{n\in\Z^+}$ by applying a series of transformations to the LDP for $(\Pi^n_W)_{n\in\Z^+}$ (which we assume below). In Section \ref{Section Noise LDP} we outline an example of a model of the noise which satisfies these assumptions (refer in particular to Theorem \ref{Theorem PinW}). In particular, the condition \eqref{eq: a limit} is proved in Lemma \ref{exp bound Wn} in Section \ref{Section Noise LDP} .
\begin{assumption}\label{Assumption Noise LDP}
 The series of laws $(\Pi_W^n)_{n\in \Z^+}$ is assumed to satisfy a Large Deviation Principle with good rate function. It is assumed that
\begin{equation}\label{eq: a limit}
\lim_{a\to\infty}\lsup{n}\frac{1}{|V_n|}\log\Prob\left(\sum_{j\in V_n}\norm{W^{n,j}}_T > a |V_n|\right) = -\infty.
\end{equation}
\end{assumption}

In many neural models, such as the Fitzhugh-Nagumo model in Section \ref{Section Example}, the internal dynamics term $\mathfrak{b}_s$ is not Lipschitz. In particular, $\mathfrak{b}_s$ is usually strongly decaying when the activity is greatly elevated, so that $\mathfrak{b}_s$ always acts to restore the neuron to its resting state. This decay is necessary in order for the neurons to exhibit their characteristic `spiking' behaviour. The following assumptions can accommodate this non-Lipschitz behaviour.
\begin{assumption}\label{Assumption Absolute Bound}
Assume that $\mathfrak{b}_t$ is continuous on $[0,T]\times\mathcal{T}$, that it is \newline$\mathcal{B}(\R) / \mathcal{B}\big(\mathcal{C}([0,t],\R)\big)$ measurable and that for each positive constant $A$,
\begin{equation}
\sup_{t\in [0,T], \lbrace X \in \T: \norm{X}_T \leq A\rbrace}|\mathfrak{b}_t(X)| < \infty.
\end{equation}
There exists a positive constant $\tilde{C}$ such that if $Z^j_t \geq 0$, then
\begin{equation*}
\mathfrak{b}_t(Z^j)  \leq \tilde{C}\norm{Z^j}_t
\end{equation*}
and if $Z^j_t \leq 0$, then
\begin{equation*}
\mathfrak{b}_t(Z^j)  \geq -\tilde{C}\norm{Z^j}_t.
\end{equation*}
If $X^j_t \geq Z^j_t$, then
\begin{equation*}
\mathfrak{b}_t(X^j) - \mathfrak{b}_t(Z^j) \leq \tilde{C}\norm{X^j - Z^j}_t,
\end{equation*}
and if $X^j_t \leq Z^j_t$, then
\begin{equation*}
\mathfrak{b}_t(X^j) - \mathfrak{b}_t(Z^j) \geq -\tilde{C}\norm{X^j - Z^j}_t.
\end{equation*}
\end{assumption}

The interactions $\Lambda^k_s(\cdot,\cdot)$ are also typically nonlinear. See for example the model of the interactions in \cite{baladron-fasoli-etal:12b}, and also the example model in Section \ref{Section Example}. Unlike in mean-field models, the interaction between any two neurons is independent of the size of the network. However the interactions must ultimately decay as the lattice distance between the pre and post synaptic neurons increases. We use the positive constants $(\kappa^k)_{k\in\Z^d}$ to bound the interaction terms $\Lambda^k_s$. The constants must satisfy
\begin{equation}\label{eq: kappa sum}
\sum_{k\in \Z^d} \kappa^k := \kappa_* < \infty,
\end{equation}
as well as the following assumptions. We assume that $\kappa^k > 0$ for all $k\in \Z^d$. We also assume that if $j(p) = \pm k(p)$ for all $p\in 1,\ldots,d$, then 
\begin{equation}\label{eq kappa jk}
\kappa^j = \kappa^k.
\end{equation}
The following assumptions on the interactions are the reason why we work in $\T^{\Z^d}_\lambda$ (rather than for example $\T^{\Z^d}$ endowed with the cylindrical topology). We will choose the weights $(\lambda^j)_{j\in\Z^d}$ carefully so that, in a certain sense, they dominate the bounds $(\kappa^k)_{k\in\Z^d}$ on the interaction strength. This is the content of Lemma \ref{Lemma Switch Inequality} further below.

\begin{assumption}\label{Assumption NonUniform}
Assume that for all $k\in \Z^d$, $\Lambda^k_s(\cdot,\cdot)$ is continuous on $[0,T]\times \T \times \T$, and for each $s\in [0,T]$, $\Lambda^k_s$ is $\mathcal{B}(\R) / \mathcal{B}\big(\mathcal{C}([0,s],\R)\big)\times \mathcal{B}\big(\mathcal{C}([0,s],\R)\big)$-measurable. For all $U,X,Z \in \T$ and $k\in\Z^d$
\begin{align*} 
\big|\Lambda^{k}_t(U,X) - \Lambda_t^{k}(U,Z)\big| &\leq \kappa^k \norm{X-Z}_t.\\
\big| \Lambda^{k}_t(U,X) - \Lambda_t^{k}(Z,X)\big| &\leq \kappa^k\norm{U-Z}_t.
\end{align*}
We assume the following absolute bound on the weights,
\begin{equation*}
\big|\Lambda^k_t(Z^j,Z^{j+k})\big| \leq \kappa^k \big(1+\norm{Z^j}_t + \norm{Z^{j+k}}_t\big).
\end{equation*}
\end{assumption}
Let $\bar{\kappa}_n = \sum_{k\notin V_n}\kappa^k$. By \eqref{eq: kappa sum}, $\bar{\kappa}_n \to 0$ as $n\to \infty$. For notational ease we assume that $\tilde{C}+\kappa_* \leq C$. We may therefore infer the following identities directly from the above assumptions. If $Z^j_t \geq 0$,
\begin{equation}\label{eq bound 1 start}
\mathfrak{b}_t(Z^j) + \sum_{k\in\Z^d}\Lambda_t^k(Z^j,Z^{j+k}) \leq C \norm{ Z^j}_t + \kappa_* + \sum_{k\in\Z^d}\kappa^k \norm{Z^{j+k}}_t,
\end{equation}
and if $Z^j_t \leq 0$,
\begin{equation}\label{eq bound 2 start}
\mathfrak{b}_t(Z^j) + \sum_{k\in\Z^d}\Lambda_t^k(Z^j,Z^{j+k}) \geq -C \norm{ Z^j}_t -\kappa_*  - \sum_{k\in\Z^d}\kappa^k\norm{Z^{j+k}}_t.
\end{equation}
Similarly, if $X^j_t - Z^j_t \geq 0$, then
\begin{equation}\label{eq bound 3 start}
\mathfrak{b}_t(X^j) - \mathfrak{b}_t(Z^j) + \sum_{k\in\Z^d}\big(\Lambda_t^k(X^j,X^{j+k})-\Lambda_t^k(Z^j,X^{j+k})\big) \leq C \norm{ X^j - Z^j}_t.
\end{equation}
If $X^j_t - Z^j_t \leq 0$, then
\begin{equation}\label{eq bound 4 start}
\mathfrak{b}_t(X^j) - \mathfrak{b}_t(Z^j) + \sum_{k\in\Z^d}\big(\Lambda_t^k(X^j,X^{j+k})-\Lambda_t^k(Z^j,X^{j+k})\big) \geq -C \norm{X^j - Z^j}_t.
\end{equation}
\section{Proofs}\label{Section Proofs}
We start by carefully defining the weights $(\lambda^j)_{j\in\Z^d}$. The weights must be chosen so that they `commute' with the bounds $\kappa^k$ of Section \ref{Section Assumptions}. This is the main content of Lemma \ref{Lemma Switch Inequality}. After this we will prove Theorem \ref{Theorem Main LDP}, using \cite[Corollary 4.2.21]{dembo-zeitouni:97}. We will be able to use this result because the maps $\Psi^n$ and $\Psi$ which map the noise $W^n$ to the neural activity $U$, as defined in \eqref{eq:fundamentalmult 2}-\eqref{eq:fundamentalmult 3}, are uniformly Lipschitz with respect to the norm on $\T^{\Z^d}_\lambda$.

We denote the Fourier transform with a tilde; hence we write for example $\tilde{\kappa}(\theta) = \sum_{j\in\Z^d}\exp\left(-i\langle \theta,j\rangle\right)\kappa^{j}$ (for $\theta \in [-\pi,\pi]^d$). It follows from \eqref{eq kappa jk} that $\tilde{\kappa}(\theta) \in \R$. For $\theta \in [-\pi,\pi]^d$, and recalling that $\kappa_* = \sum_{j\in\Z^d}\kappa^j $, define
\begin{align}
\tilde{\lambda}(\theta) &= h\left(2\kappa_* - \tilde{\kappa}(\theta)\right)^{-1}, \text{ and let }\label{eq lambda j defn 0}\\
\lambda^j &= \frac{1}{(2\pi)^d}\int_{[-\pi,\pi]^d}\exp\left(i\langle \theta,j\rangle\right)\tilde{\lambda}(\theta)d\theta,\label{eq lambda j defn}
\end{align}
assuming that $h$ is scaled such that
\begin{equation}\label{eq scale h lambda}
\frac{1}{(2\pi)^d}\int_{[-\pi,\pi]^d}\tilde{\lambda}(\theta)d\theta = 1.
\end{equation}
Note that the Fourier Series decomposition means that for $\theta \in [-\pi,\pi]^d$,
$\tilde{\lambda}(\theta) = \sum_{j\in\Z^d}\exp\left(-i\langle \theta,j\rangle\right)\lambda^{j}$.

The one-dimensional version of inequality \eqref{eq: lambdakappa2} in the following Lemma has been proved in \cite[Lemma 4.2]{shiga-shimizu:80}.
\begin{lemma}\label{Lemma Switch Inequality}
For all $j\in\Z^d$, $\lambda^j > 0$, and
\begin{equation}
\sum_{k\in\Z^d}\lambda^{j-k}\kappa^k \leq 2\kappa_*\lambda^j. \label{eq: lambdakappa2}
\end{equation}
Finally,
\begin{equation}\label{eq total lambda}
\sum_{j\in\Z^d}\lambda^j = 1.
\end{equation}
\end{lemma}
\begin{proof}
We start with the first statement. Through a Taylor Expansion, we see that
\begin{equation}\label{eq Taylor 1}
\tilde{\lambda}(\theta) = \frac{h}{2\kappa_*}\sum_{k=0}^\infty 2^{-k}\kappa_*^{-k} \left(\tilde{\kappa}(\theta)\right)^k.
\end{equation}
Now for any two functions $\tilde{f},\tilde{g}$ with absolutely convergent Fourier Series, if for all $j\in\Z^d$ $f^j > 0$ and $g^j > 0$, then each Fourier coefficient of the multiplication $(\tilde{f}\tilde{g})$ is also strictly greater than zero. This is because of the convolution formula, with the $j^{th}$ Fourier coefficient of $\tilde{f}\tilde{g}$ equal to
\[
\sum_{k\in\Z^d} f^{j-k}g^k > 0.
\]
Hence the Fourier Coefficients of each term in \eqref{eq Taylor 1} of the form $\tilde{\kappa}(\theta)^k$ are greater than zero (because $\kappa^k > 0$ for all $k\in\Z^d$). This means that the Fourier Coefficients of $\tilde{\lambda}$ are greater than zero. 

The second identity \eqref{eq: lambdakappa2} follows from taking the $j^{th}$ Fourier Coefficient of both sides of the equation
\begin{equation*}
-\tilde{\lambda}(\theta) \left(2\kappa_* - \tilde{\kappa}(\theta)\right)= -h,
\end{equation*}
which derives from \eqref{eq lambda j defn 0}. Upon doing this, we find that $\sum_{k\in\Z^d}\lambda^{j-k}\kappa^k - 2\kappa_*\lambda^j$ is the $j^{th}$ Fourier coefficient of $-h$, which by definition is strictly less than zero if $j=0$, else otherwise is zero.

The final identity \eqref{eq total lambda} follows directly from \eqref{eq scale h lambda}.
\end{proof}

Let $\bar{\T}_\lambda^{\Z^d} = \lbrace w \in \T^{\Z^d}_\lambda | w^j_0 = 0\rbrace$. Define $\Psi^n,\Psi: \bar{\T}^{\Z^d}_\lambda \to \T^{\Z^d}_\lambda$ as follows. Writing $\Psi^n(w) := X$ and $\Psi(w) := Z$, for any $j\in\Z^d$ and $t\in [0,T]$,
\begin{align}\label{eq:fundamentalmult 2}
X^{j}_t &:= U_{ini} + \int_0^t \bigg(\mathfrak{b}_s(X^{j}) + \sum_{k\in V_n}\Lambda^{k}_s(X^j,X^{j+k})\bigg)ds + w^j_{t} \\
Z^{j}_t &:= U_{ini} + \int_0^t \bigg(\mathfrak{b}_s(Z^{j}) + \sum_{k\in \Z^d}\Lambda^{k}_s(Z^j,Z^{j+k})\bigg)ds + w^j_{t}.\label{eq:fundamentalmult 3}
\end{align}
\begin{lemma}\label{lem: Psin Psi}
$\Psi,\Psi^n: \bar{\T}^{\Z^d}_\lambda \to \T^{\Z^d}_\lambda$ are well-defined and unique.
\end{lemma}
\begin{proof}
Fix $w \in \bar{\T}^{\Z^d}_{\lambda}$. We prove the existence of $\Psi(w)$ satisfying \eqref{eq:fundamentalmult 3} by using periodic approximations of $w$ and Lemma \ref{Lemma : Psi periodic solution}. The uniqueness is a direct consequence of Lemma \ref{Lemma Psi Lipschitz}. The proof for $\Psi^n$ is analogous.

Fix $q\in\Z^+$ and let $p\in\Z^+$ be such that $\sum_{j\notin V_p}\lambda^j \norm{w^j}_T^2 \leq \frac{1}{4q}$. Let $K=\sup_{j\in V_p}\norm{w^j}^2_T$ and $m\in\Z^+$ be such that $K\sum_{j\notin V_m}\lambda^j \leq \frac{1}{4q}$. Define $\tilde{w}(q) \in \T^{\Z^d}_\lambda$ to be such that
\begin{align*}
\tilde{w}(q)^j &= w^j \text{ for $j\in V_p$} \\
\tilde{w}(q)^j &= 0 \text{ for $j\in V_m / V_p$} \\
\tilde{w}(q)^j &= \tilde{w}(q)^{j\modd V_m} \text{ otherwise.}
\end{align*}
We observe that
\begin{align}
\sum_{j\in \Z^d}\lambda^j\norm{w^j - \tilde{w}(q)^j}_T^2 &\leq \sum_{j\notin V_p}\lambda^j\big(\norm{w^j}_T + \norm{\tilde{w}(q)^j}_T\big)^2\nonumber \\
&\leq 2\sum_{j\notin V_p}\lambda^j\big(\norm{w^j}^2_T + \norm{\tilde{w}(q)^j}_T^2\big) \nonumber\\
&\leq 2\bigg(\frac{1}{4q}+ \sum_{j\notin V_m}\lambda^j K\bigg)\nonumber \\
&\leq q^{-1}.\label{eq: epsilon bound 1}
\end{align}
By Lemma \ref{Lemma : Psi periodic solution}, there exists a solution $Z^{(q)} := \Psi(\tilde{w}(q))$ to \eqref{eq:fundamentalmult 3}. If $r > q$, then by Lemma \ref{Lemma Psi Lipschitz}, for a positive constant $\Psi_C$,
\begin{align}\label{Eq : Z Cauchy Temp}
\norm{Z^{(q)} - Z^{(r)}}_{T,\lambda} \leq \Psi_C\norm{\tilde{w}(q) -\tilde{w}(r)}_{T,\lambda}.
\end{align}
Since, by \eqref{eq: epsilon bound 1},
\begin{align*}
\norm{\tilde{w}(q) -\tilde{w}(r)}_{T,\lambda} &\leq \norm{\tilde{w}(q) -w}_{T,\lambda} + \norm{w-\tilde{w}(r)}_{T,\lambda}\\
&\leq \frac{2}{\sqrt{q}},
\end{align*}
we may infer from \eqref{Eq : Z Cauchy Temp} that the sequence $(Z^{(q)})_{q\in\Z^+}$ converges to a limit $Z^* \in \T^{\Z^d}_\lambda$ as $q \to \infty$. To finish, it suffices for us to prove that for all $t\in [0,T]$ and $j\in\Z^d$,
\begin{multline}
\lim_{q\to \infty}\int_0^t \mathfrak{b}_s(Z^{(q)j}) + \sum_{k\in\Z^d}\Lambda^k_s(Z^{(q)j},Z^{(q)j+k})ds\\
=\int_0^t \mathfrak{b}_s(Z^{*j}) + \sum_{k\in\Z^d}\Lambda^k_s(Z^{*j},Z^{*j+k})ds.\label{eq: to show 1}
\end{multline}
Now since $Z^{(q)}$ converges as $q \to \infty$, for each $j$, there must be some positive constant $K$ such that
\begin{equation*}
\sup_{q \in\Z^+}\norm{Z^{(q)j}}_T  \leq K.
\end{equation*}
But by Assumption \ref{Assumption Absolute Bound}, this means that $\sup_{q\in\Z^+,t\in [0,T]}\big|\mathfrak{b}_t(Z^{(q)j})\big|  < \infty$. Since $\mathfrak{b}_t$ is continuous, $\mathfrak{b}_t(Z^{(q)j}) \to \mathfrak{b}_t(Z^{*j})$, so that we may conclude by the dominated convergence theorem that
\[
\int_0^t \mathfrak{b}_s(Z^{(q)j}) ds \to \int_0^t \mathfrak{b}_s(Z^{*j}) ds. 
\]
For the other terms, we see that
\begin{multline*}
\bigg|\int_0^t\sum_{k\in\Z^d}\big(\Lambda^k_s(Z^{(q)j},Z^{(q)j+k})-\Lambda^k_s(Z^{*j},Z^{*j+k})\big)ds\bigg|\\
\leq \int_0^t \sum_{k\in\Z^d}\bigg|\Lambda^k_s(Z^{(q)j},Z^{(q)j+k})-\Lambda^k_s(Z^{*j},Z^{(q)j+k})\bigg|\\
+\sum_{k\in\Z^d}\bigg|\Lambda^k_s(Z^{*j},Z^{(q)j+k})-\Lambda^k_s(Z^{*j},Z^{*j+k})\bigg|ds\\
\leq \int_0^t \sum_{k\in\Z^d}\kappa^k\bigg(\norm{Z^{(q)j}-Z^{*j}}_s + \norm{Z^{(q)j+k}-Z^{*j+k}}_s\bigg)ds.
\end{multline*}
Summing over $j$,
\begin{align*}
\sum_{j\in\Z^d}&\lambda^j\bigg|\int_0^t\sum_{k\in\Z^d}\big(\Lambda^k_s(Z^{(q)j},Z^{(q)j+k})-\Lambda^k_s(Z^{*j},Z^{*j+k})\big)ds\bigg|\\
&\leq \int_0^t\bigg( \kappa_*\sum_{j\in\Z^d}\lambda^j\norm{Z^{(q)j}-Z^{*j}}_s  + \sum_{j,k\in\Z^d}\lambda^j\kappa^k \norm{Z^{(q)j+k}-Z^{*j+k}}_s \bigg)ds\\
&\leq \int_0^t \bigg( 3\kappa_*\sum_{j\in\Z^d}\lambda^j\norm{Z^{(q)j}-Z^{*j}}_s \bigg)ds\\
&\leq \int_0^t 3\kappa_*\norm{Z^{(q)j}-Z^{*j}}_{s,\lambda} ds,
\end{align*}
by Lemma \ref{Lemma Switch Inequality} and the Cauchy-Schwarz Inequality. Since $\norm{Z^{(q)j}-Z^{*j}}_{s,\lambda} \to 0$ as $q \to \infty$, it must be that as $q\to\infty$, for each $j\in\Z^d$,
\[
\int_0^t\sum_{k\in\Z^d}\big(\Lambda^k_s(Z^{(q)j},Z^{(q)j+k})-\Lambda^k_s(Z^{*j},Z^{*j+k})\big)ds \to 0.
\]
We have thus established \eqref{eq: to show 1}.
\end{proof}

\begin{lemma}\label{Lemma : Psi periodic solution}
Suppose that $w \in \bar{\T}^{\Z^d}_\lambda$ is $V_m$ periodic, i.e. $w^{k} = w^{k\modd V_m}$ for all $k\in\Z^d$. Then there exist $\Psi^n(w)$ and $\Psi(w)$ satisfying \eqref{eq:fundamentalmult 2}-\eqref{eq:fundamentalmult 3}.
\end{lemma}
\begin{proof}
We prove the result for $\Psi(w)$ (the case $\Psi^n(w)$ is analogous). It suffices for us to show existence to the following finite-dimensional ODE: for $j\in V_m$,
\begin{equation}\label{eq u soln 1}
U^{j}_t := U_{ini} + \int_0^t \bigg(\mathfrak{b}_s(U^{j}) + \sum_{k\in \Z^d}\Lambda^{k}_s(U^j,U^{(j+k)\modd V_m})\bigg)ds + w^j_{t}.
\end{equation}
Once existence to the above equation has been shown, we may define $\Psi(w)$ to be the $V_m$-periodic extension of this solution, i.e. $\Psi(w)^j := U^{j\modd V_m}$. 

Define $Y^j_t = U^j_t - w^j_t$. It can be seen that the existence of a solution to \eqref{eq u soln 1} is equivalent to the existence of a solution to the finite-dimensional differential equation, for all $j\in V_m$,
\begin{equation}\label{eq Y soln 1}
\frac{d}{dt}Y^{j}_t = \mathfrak{b}_t(Y^{j} + w^j) + \sum_{k\in \Z^d}\Lambda^{k}_t\big(Y^j+w^j,Y^{(j+k)\modd V_m}+w^{(j+k)\modd V_m}\big),
\end{equation}
such that $Y^j_0 = U_{ini}$. Suppose for the moment that there were to exist a solution to \eqref{eq Y soln 1} over some time interval $[0,\alpha]$. Then $\norm{Y^j}_t \leq \norm{U^j}_t + \norm{w^j}_t$ and therefore by Lemma \ref{Lemma Bound Psin}
\begin{align}\label{eq: Omega defn1}
\sum_{j\in V_m}\norm{Y^j}_\alpha \leq&  \zeta_\alpha \\
\zeta_\alpha =& \bigg(|V_m|\big(|U_{ini}|+\alpha\kappa_*\big) +3\sum_{j\in V_m}\norm{w^j}_\alpha\bigg)\exp\big((C+\kappa_*)\alpha\big).\nonumber
\end{align}
We will now use the generalisation of the Cauchy-Peano Theorem in Lemma \ref{Lem General Cauchy Peano} to prove the existence of a solution to \eqref{eq Y soln 1}. Let $\Omega_t$ be the set of all $Y\in \mathcal{C}\big([0,t],\R^{V_m}\big)$ satisfying $\norm{Y^j}_t \leq 2\zeta_t$ for all $j\in V_m$. It may be observed using the triangle inequality that
\begin{equation*}
\left\lbrace Y \in\mathcal{C}\big([0,t],\R^{V_m}\big): \norm{Y^j - U_{ini}}_t \leq \zeta_t\right\rbrace \subseteq \Omega_t.
\end{equation*}
Let $M\in [0,\infty]$, be
\begin{multline}\label{eq:M bound}
M := \sup_{t\in [0,T],Y \in \Omega_T,j\in V_m}\bigg|\mathfrak{b}_t(Y^{j} + w^j) \\+ \sum_{k\in \Z^d}\Lambda^{k}_t\big(Y^j+w^j,Y^{(j+k)\modd V_m}+w^{(j+k)\modd V_m}\big)\bigg|.
\end{multline}
We claim that $M < \infty$. This is because for each $j\in V_m$, $\sup_{Y\in\Omega_T}\norm{Y^j + w^j}_T < \infty$, which means that $\sup_{t\in [0,T],Y \in \Omega_T,j\in V_m}\big|\mathfrak{b}_t(Y^{j} + w^j)\big| < \infty$ (thanks to Assumption \ref{Assumption Absolute Bound}). For the other term, it may be observed that the sum over over $j$ is finite, i.e.
\begin{align*}
&\sum_{j\in V_m}\bigg| \sum_{k\in \Z^d}\Lambda^{k}_t(Y^j+w^j,Y^{(j+k)\modd V_m}+w^{(j+k)\modd V_m})\bigg| \\ &\leq \sum_{j\in V_m,k\in\Z^d}\kappa^k\big( \norm{Y^j + w^j}_T + \norm{Y^{(j+k)\modd V_m}+w^{(j+k)\modd V_m}}_T + 1\big) \\
&=\kappa_* | V_m| + 2\kappa_*\sum_{j\in V_m}\norm{Y^{j}+w^j}_T < \infty.
\end{align*}
We thus see that $M < \infty$. We may therefore use Lemma \ref{Lem General Cauchy Peano} to conclude that there exists a solution $Y^*$ to \eqref{eq Y soln 1} over the time interval $[0,\alpha]$, where $\alpha = \min\big(T,\zeta_T / M\big)$. Furthermore, from \eqref{eq: Omega defn1}, this solution must satisfy
\begin{equation}
\sum_{j\in V_m}\norm{Y^{*j}}_{\alpha} \leq \zeta_\alpha,
\end{equation}
so that $Y^* \in \Omega_t$. If $\alpha \neq T$ we may continue this process iteratively. It follows from the triangle inequality that
\begin{multline*}
\bar{\Omega}_{2\alpha} := \left\lbrace Z \in \mathcal{C}([0,2\alpha],\R^{V_m}): Z^j_t = Y^{*j}_t\text{ for all }t\in [0,\alpha]\right. \\ \left.\text{ and }\sup_{t\in [\alpha,T] , j\in V_m}\big|Z^j_t - Y^{*j}_\alpha\big| \leq \zeta_T\right\rbrace \subseteq \Omega_T.
\end{multline*}
Note also that $\mathfrak{b}_t$ and $\Lambda_t$, when restricted to $\bar{\Omega}_{2\alpha}$, may be interpreted as continuous functions on $\mathcal{C}([\alpha,T],\R^{V_m})$. Hence by Lemma \ref{Lem General Cauchy Peano} (we replace the interval $[0,T]$ in this lemma by $[\alpha,T]$), there exists a solution $Z^* \in \bar{\Omega}_{2\alpha}$ which satisfies \eqref{eq Y soln 1} for all $t\in [\alpha,2\alpha]$.

We may continue this process iteratively to find a solution over the entire time interval $[0,T]$. The reason that this works is that the solution must always be in $\Omega$, and the bound in \eqref{eq:M bound} is over all of $\Omega$, which is why the increment in the time interval is always $\alpha$.
\end{proof}
We use the following generalisation of the Cauchy-Peano Existence theorem. 
\begin{lemma}\label{Lem General Cauchy Peano}
Let $\Omega$ be a closed subset of $\T^{V_m}$ such that
\[
\lbrace X \in \T^{V_m}: X^j_0 = X_{ini} \text{ and }\norm{X^j-X_{ini}}_T \leq \beta \text{ for all }j\in V_m\rbrace \subseteq \Omega.
\]
Suppose that $f \in \mathcal{C}\big([0,T]\times \Omega,\R^{V_m}\big)$ is such that for all $t\in [0,T]$, $f(t,\cdot)$ is $\mathcal{B}\big(\R^{V_m} \big) / \mathcal{B}(\mathcal{C}([0,t],\R^{V_m}))$ measurable and
\begin{equation}
\sup_{t\in [0,T],X \in \Omega,j\in V_m}|f(t,X)^j| = M.
\end{equation}
Then there exists $X\in \Omega$ such that for all $t\in [0,\alpha]$ (where $\alpha = min(T,\beta / M)$) and $j\in V_m$,
\begin{equation*}
X^j_t = X_{ini} + \int_0^t f(s,X)^j ds.
\end{equation*}
\end{lemma}
\begin{proof}
Divide $[0,\alpha]$ into $k+1$ points $t_0:=0, t_1 := \alpha / k,t_2:=2\alpha/k,\ldots,t_k:=\alpha$. Let $\Omega' = \lbrace X \in \mathcal{C}([0,\alpha],\R^{V_m}): \norm{X^j-X_{ini}}_\alpha \leq \beta \text{ for all }j\in V_m\rbrace$. Let $X^{(k)} \in \mathcal{C}([0,\alpha],\R^{V_m})$ be such that 
\begin{align*}
X^{(k)j}_0 &= X_{ini} \\
X^{(k)j}_t &= X^{(k)j}_s + (t-s)f(s,X^{(k)})^j,
\end{align*}
where $s= \sup\lbrace t_q: t_q\leq t\rbrace$. We note that $X^{(k)} \in \Omega'$, because
\begin{align*}
\big|X^{(k)j}_{t_p} - X_{ini}\big| &\leq \sum_{l=1}^p \big| X^{(k)j}_{t_l} - X^{(k)j}_{t_{l-1}}\big| \\
 & \leq \sum_{l=1}^p \frac{\alpha}{k}\norm{f(t_{l-1},X^{(k)})^j} \\
 &\leq \alpha M \leq \beta.
\end{align*}
We see that in general $\norm{X^{(k)j}}_\alpha \leq \beta$, and for $t\in [t_l,t_{l+1}]$, $\big|X^{(k)j}_t - X^{(k)j}_{t_l}\big| \leq \frac{M\alpha}{k}$. This means that the sequence $\lbrace X^{(k)}\rbrace_{k=1}^\infty$ is equicontinuous, and therefore compact by the Arzela-Ascoli Theorem. Thus there exists a subsequence $(k_p)_{p=1}^\infty$ and $X^* \in \mathcal{C}([0,\alpha],\R^m)$ such that for each $j\in V_m$, $X^{(k_p)j}_t \to X^{*j}_t$ uniformly in $t$. We have that
\begin{equation*}
X^{(k_p)j}_t = X_{ini} + \int_0^t f(s^{(k_p)},X^{(k_p)})^j ds,
\end{equation*}
where $s^{(k_p)} = \sup\lbrace t_q: t_q\leq s\rbrace$, the supremum being taken over the partition with $k_p + 1$ points. Now as $p\to\infty$, $s^{(k_p)} \to s$, and by the continuity of $f$, $f(s^{(k_p)},X^{(k_p)})\to f(s,X^*)$. Since $f$ is bounded on $[0,T]\times \Omega'$, by the dominated convergence theorem, $\int_0^t f(s^{(k_p)},X^{(k_p)})^j ds \to \int_0^t f(s,X^*)^j ds$. We thus see that as $p\to \infty$,
\begin{equation*}
X^{*j}_t = X_{ini} + \int_0^t f(s,X^{*})^j ds.
\end{equation*}
\end{proof}
\begin{lemma}\label{Lemma Psi Lipschitz}
There exists a constant $\Psi_C$ satisfying the following. If $\Psi^n(w)$ and $\Psi^n(v)$ are solutions to \eqref{eq:fundamentalmult 2}, then
\[
\norm{\Psi^n(w)-\Psi^n(v)}_{T,\lambda} \leq \Psi_C\norm{w-v}_{T,\lambda}.
\]
If $\Psi(w)$ and $\Psi(v)$ are solutions to \eqref{eq:fundamentalmult 3} then
\[
\norm{\Psi(w)-\Psi(v)}_{T,\lambda} \leq \Psi_C\norm{w-v}_{T,\lambda}.
\]
\end{lemma}
\begin{proof}
We prove the result for $\Psi$ (the other case follows analogously). Let $v,w \in \T^{\Z^d}_\lambda$, and write $X = \Psi(w)$ and $Z = \Psi(v)$. 

Let $[\tau,\gamma] \subset [0,T]$ be such that $X^j_\tau = Z^j_\tau$, and $X^j_t - Z^j_t$ is of the same sign for all $t\in [\tau,\gamma]$. We see that for all $t\in [\tau,\gamma]$,
\begin{align*}
&X^j_t - Z^j_t\\ &= \int_{\tau}^t\bigg( \mathfrak{b}_s(X^j) - \mathfrak{b}_s(Z^j) + \sum_{k\in\Z^d}\Lambda^k_s(X^j,X^{j+k}) - \Lambda^k_s(Z^j,Z^{j+k})\bigg)ds \\ &+ w^j_t-w^j_{\tau} - v^j_t+ v^j_{\tau} + X^j_{\tau} - Z^j_{\tau}.
\end{align*}
Hence for $r\in [\tau,\gamma]$, since $X^j_\tau = Z^j_\tau$,
\begin{align*}
&\sup_{t\in [\tau,r]}\left| X^j_t - Z^j_t\right| \\&\leq \sup_{t\in [\tau,r]}\bigg|\int_{\tau}^t\bigg( \mathfrak{b}_s(X^j) - \mathfrak{b}_s(Z^j) + \sum_{k\in\Z^d}\Lambda^k_s(X^j,X^{j+k}) - \Lambda^k_s(Z^j,X^{j+k}) \\
&+ \sum_{k\in\Z^d}\Lambda^k_s(Z^j,X^{j+k}) - \Lambda^k_s(Z^j,Z^{j+k}\bigg)ds \bigg|+ 2\norm{w^j-v^j}_T \\
&\leq \int_{\tau}^r\bigg(C\norm{X^j-Z^j}_s + \sum_{k\in\Z^d}\kappa^k \norm{X^{j+k}-Z^{j+k}}_s\bigg)ds + 2\norm{w^j-v^j}_T,
\end{align*}
since by assumption $X^j_t - Z^j_t$ is of the same sign, which allows us to use \eqref{eq bound 3 start} and \eqref{eq bound 4 start}. We thus see that for any $t\in [0,T]$,
\begin{equation*}
\norm{X^j-Z^j}_t \leq 2\norm{w^j-v^j}_T + \int_0^t\bigg(C \norm{X^j - Z^j}_s + \sum_{k\in\Z^d}\kappa^k \norm{X^{j+k}-Z^{j+k}}_s\bigg)ds.
\end{equation*}
Through an application of Gronwall's Inequality to the above, for all $t\in [0,T]$,
\begin{equation*}
\norm{X^j-Z^j}_t \leq \bigg( 2\norm{w^j - v^j}_T + \int_0^t\sum_{k\in\Z^d}\kappa^k \norm{X^{j+k}-Z^{j+k}}_s ds \bigg)\exp\big( CT \big).
\end{equation*}
Since $(A+B)^2 \leq 2(A^2 + B^2)$, we thus see that through two applications of Jensen's Inequality,
\begin{align*}
&\sum_{j\in\Z^d}\lambda^j\norm{X^j - Z^j}_t^2 \\ &\leq 2\exp\big( 2CT\big) \bigg[ 4\sum_{j\in\Z^d}\lambda^j \norm{w^j - v^j}_t^2 + \sum_{j\in\Z^d}\lambda^j\bigg(\int_0^t\sum_{k\in\Z^d}\kappa^k \norm{X^{j+k}-Z^{j+k}}_s ds\bigg)^2\bigg] \\
&\leq 2\exp\big( 2CT\big) \bigg[ 4\sum_{j\in\Z^d}\lambda^j \norm{w^j - v^j}_t^2 + t\int_0^t \sum_{j\in\Z^d}\lambda^j\big(\sum_{k\in\Z^d}\kappa^k \norm{X^{j+k}-Z^{j+k}}_s\big)^2 ds\bigg] \\
&\leq 2\exp\big( 2CT\big) \bigg[ 4\sum_{j\in\Z^d}\lambda^j \norm{w^j - v^j}_t^2 + t\kappa_*\int_0^t \sum_{j,k\in\Z^d}\lambda^j\kappa^k \norm{X^{j+k}-Z^{j+k}}_s^2 ds\bigg] \\
 &\leq 2\exp\big( 2CT\big) \bigg[ 4\sum_{j\in\Z^d}\lambda^j \norm{w^j - v^j}_t^2 +2t\kappa_*^2\int_0^t \sum_{j\in\Z^d}\lambda^j\norm{X^{j}-Z^{j}}_s^2 ds\bigg],
 \end{align*}
 where we have used Lemma \ref{Lemma Switch Inequality}. We apply Gronwall's Inequality to the above to find that there exists a positive constant $\Psi_c$,  with\newline$\Psi_c^2 := 8\exp\big(4T^2\kappa_*^2\exp(2CT) + 2CT\big)$, and such that for all $t\in [0,T]$,
 \begin{equation*}
 \sum_{j\in\Z^d}\lambda^j\norm{X^j - Z^j}_t^2 \leq \Psi_C^2  \sum_{j\in\Z^d}\lambda^j\norm{w^j - v^j}_t^2.
 \end{equation*}
 This gives us the lemma.
\end{proof}
\begin{lemma}\label{Lemma Bound Psin}
Suppose that $w\in \bar{\T}^{\Z^d}_\lambda$ is $V_m$-periodic, i.e. $w^j = w^{j\modd V_m}$. If $\Psi^n(w)$ is a solution to \eqref{eq:fundamentalmult 2} over some time interval $[0,\alpha] \subseteq [0,T]$, then
\begin{equation*}
\sum_{j\in V_m}\norm{\Psi^n(w)^j}_{\alpha} \leq \exp\big((C+\kappa_*)\alpha\big) \bigg(|V_m| \big(|U_{ini}| + \alpha\kappa_*\big) +2 \sum_{j\in V_m}\norm{w^j}_{\alpha}\bigg).
\end{equation*}
If $\Psi(w)$ is a solution to \eqref{eq:fundamentalmult 3}, then 
\begin{equation*}
\sum_{j\in V_m}\norm{\Psi(w)^j}_{\alpha} \leq \exp\big((C+\kappa_*)\alpha\big) \bigg(|V_m| \big(|U_{ini}| + \alpha\kappa_*\big) +2 \sum_{j\in V_m}\norm{w^j}_{\alpha}\bigg).
\end{equation*}
\end{lemma}
\begin{proof}
We prove the first of these results (the other proof is analogous). Write $Z = \Psi^n(w)$. We may suppose that $\norm{Z^j}_T > |U_{ini}|$ for some $j$, because otherwise the lemma holds trivially. For some $j\in\Z^d$, let $[\tau,\gamma] \subseteq [0,\alpha]$ be such that $Z^j_t$ has the same sign for all $t\in [\tau,\gamma]$ and $|Z^j_\tau| = |U_{ini}|$. Then since
\begin{align*}
Z^j_t = Z^j_\tau + \int_\tau^t\bigg( \mathfrak{b}_s(Z^j) + \sum_{k\in V_n}\Lambda^k_s(Z^j,Z^{j+k})\bigg)ds + w^j_t - w^j_{\tau},
\end{align*}
it follows from \eqref{eq bound 1 start}-\eqref{eq bound 2 start} that for any $r\in [\tau,\gamma]$,
\begin{align*}
&\sup_{t \in [\tau,r]}|Z^j_t| \leq |U_{ini}| + \sup_{t\in [\tau,r]}\int_{\tau}^t \bigg(C\norm{Z^j}_s + \sum_{k\in\Z^d}\kappa^k \norm{Z^{j+k}}_s  \bigg)ds\\ &\;\;\;\;\;+ \big|w^j_t \big| + \big|w^j_\tau\big| + \kappa_*(r-\tau) \\
&\leq  |U_{ini}|  + \int_{\tau}^r \bigg(C\norm{Z^j}_s + \sum_{k\in\Z^d}\kappa^k \norm{Z^{j+k}}_s \bigg)ds+ 2\norm{w^j}_r+(r-\tau)\kappa_*.
\end{align*}
We thus see that
\begin{equation*}
\norm{Z^j}_r \leq |U_{ini}|+r\kappa_*  + \int_{0}^r \bigg(C\norm{Z^j}_s + \sum_{k\in\Z^d}\kappa^k \norm{Z^{j+k}}_s \bigg)ds+ 2\norm{w^j}_r,
\end{equation*}
and therefore
\begin{align*}
\sum_{j\in V_m}\norm{Z^j}_r &\leq |V_m|\big(|U_{ini}|+r\kappa_*\big)+2\sum_{j\in V_m}\norm{w^j}_r  \\ & \; \; \;\;\;\;\;+ \int_{0}^r \bigg(C\sum_{j\in V_m}\norm{Z^j}_s + \sum_{j\in V_m,k\in\Z^d}\kappa^k \norm{Z^{j+k}}_s \bigg)ds\\
&= |V_m|\big(|U_{ini}|+r\kappa_*\big)+2\sum_{j\in V_m}\norm{w^j}_r  \\ &\; \; \;\;\;\;\;+ \int_{0}^r \bigg(C\sum_{j\in V_m}\norm{Z^j}_s + \kappa_*\sum_{j\in V_m} \norm{Z^{j}}_s \bigg)ds,
\end{align*}
since $Z^j$ is $V_m$-periodic (by Lemma \ref{Lemma Solution Equivalence}). The lemma now follows through an application of Gronwall's Inequality.
\end{proof}
The following lemma notes that $\Psi^n$ and $\Psi$ preserve the periodicity of $w$.
\begin{lemma}\label{Lemma Solution Equivalence}
Suppose that $w\in\bar{\T}^{\Z^d}_\lambda$ is $V_m$-periodic. Then for all $j\in\Z^d$,
\begin{align*}
\Psi^n(w)^j &= \Psi^n(w)^{j \modd V_m} \\
\Psi(w)^j &= \Psi(w)^{j \modd V_m} \\
\Psi^n(\tilde{W}^n)^j &= U^{j \modd V_n},
\end{align*}
where $U$ is defined in \eqref{eq:fundamentalmult}. Furthermore
\begin{equation}
\hat{\mu}^n(U) = \hat{\mu}^n(W^n) \circ (\Psi^n)^{-1}.\label{eq:equivalence of empirical measures}
\end{equation}
\end{lemma}
\begin{proof}
It follows from the definition that for all $j\in \Z^d$, $\Psi(S^j w) = S^j\Psi(w)$ and $\Psi^n(S^j w) = S^j\Psi^n(w)$. If $k\in V_m$ is such that $j\modd V_m = k$, then from the definition $S^j w = S^k w$, which gives us the first two results. \eqref{eq:equivalence of empirical measures} follows directly from this and the definition of the empirical measure.
\end{proof}
We now prove the main result: Theorem \ref{Theorem Main LDP}. The proof uses parts of \cite[Theorem 4.9]{baxter-jain:93}.
\begin{proof}[Proof of Theorem \ref{Theorem Main LDP}]
From Lemma \ref{Lemma Solution Equivalence}, the law of $\hat{\mu}^n(W^n)\circ(\Psi^n)^{-1}$ is $\Pi^n$. By Lemma \ref{Lemma Psi Lipschitz}, the maps $\mu \to \mu \circ \Psi^{-1}$ and $\mu \to \mu \circ (\Psi^n)^{-1}$ are continuous. Therefore using \cite[Corollary 4.2.41]{dembo-zeitouni:97}, and our assumption that $\Pi^n_W$ satisfies an LDP with good rate function, it suffices to prove that for any $\delta > 0$,
\begin{equation}\label{eq to prove Dec}
\lim_{n\to\infty}\frac{1}{|V_n|}\log\PP\bigg(d^{\lambda,\mathcal{P}}\bigg(\hat{\mu}^n(W^n)\circ(\Psi^n)^{-1},\hat{\mu}^n(W^n)\circ \Psi^{-1}\bigg) > \delta\bigg) = -\infty.
\end{equation}
Let $\tilde{W}^n \in \T^{\Z^d}_\lambda$ be the $V_n$-periodic interpolation of $(W^{n,j})_{j\in V_n}$ - i.e. such that $\tilde{W}^{n,k} := W^{k\modd V_n}$ for all $k\in \Z^d$. Let $X^n = \Psi(\tilde{W}^n) - \Psi^n(\tilde{W}^n)$, $Y^n = \Psi(\tilde{W}^n)$ and $Z^n = \Psi^n(\tilde{W}^n)$. For $A\in\mathcal{B}(\mathcal{T}_\lambda^{\Z^d})$, let $A^\delta = \lbrace x\in \mathcal{T}_\lambda^{\Z^d}:\norm{x-y}_{T,\lambda} \leq \delta \text{ for some }y\in A\rbrace$ be the closed blowup of $A$, and let $\mathfrak{Z}(\delta)$ be the closed blowup of $\lbrace 0\rbrace$. Then, letting $\one$ denote the indicator function, and noting that $S^j\Psi(\tilde{W}^n) = \Psi(S^j\tilde{W}^n)$ and $S^j\Psi^n(\tilde{W}) = \Psi^n(S^j\tilde{W})$ (as stated in Lemma \ref{Lemma Solution Equivalence}), we see that
\begin{align*}
\hat{\mu}^n(W^n)\circ\Psi^{-1}(A) &= \frac{1}{|V_n|}\sum_{j\in V_n}\one_A\left( S^j Z^n + S^j X^n \right) \\
&\leq \frac{1}{|V_n|}\sum_{j\in V_n}\bigg[ \one_A\left(S^j Z^n + S^j X^n\right)\one_{\mathfrak{Z}(\delta)}(S^j X^n) \\ & + \one_{\mathfrak{Z}(\delta)^c}\left(S^j X^n \right)\bigg].
\end{align*} 
Now if $S^j X^n \in \mathfrak{Z}(\delta)$ and $\big(S^j Z^n + S^j X^n\big) \in A$, then $S^j Z^n \in A^\delta$. This means that
\[
\one_A\left(S^j Z^n + S^j X^n\right)\one_{\mathfrak{Z}(\delta)}(S^j X^n) \leq \one_{A^\delta}(S^j Z^n). 
\]
We may therefore conclude that, after letting $\#\lbrace\cdot \rbrace$ denote the cardinality of a finite set,
\begin{align*}
\hat{\mu}^n(W^n)\circ\Psi^{-1}(A) \leq& \frac{1}{|V_n|}\sum_{j\in V_n}\one_{A^\delta}(S^j Z^n) + \frac{1}{|V_n|} \#\left\lbrace j\in V_n: \norm{S^j X^n}_{T,\lambda} > \delta\right\rbrace\\
=&\hat{\mu}^n(W^n)\circ(\Psi^n)^{-1}(A^\delta) + \frac{1}{|V_n|} \#\left\lbrace j\in V_n: \norm{S^j X^n}_{T,\lambda} > \delta\right\rbrace.
\end{align*}
Therefore, using the definition of the Levy-Prokhorov Metric,
\begin{multline*}
d^{\lambda,\mathcal{P}}\bigg(\hat{\mu}^n(W^n)\circ(\Psi^n)^{-1},\hat{\mu}^n(W^n)\circ \Psi^{-1}\bigg) \\ \leq \max\left\lbrace \delta, \frac{1}{|V_n|} \#\lbrace j\in V_n: \norm{S^j X^n}_{T,\lambda} > \delta\rbrace \right\rbrace. 
\end{multline*}
Hence
\begin{align*}
\PP\bigg(d^{\lambda,\mathcal{P}}\bigg(\hat{\mu}^n(W^n)\circ(\Psi^n)^{-1}&,\hat{\mu}^n(W^n)\circ \Psi^{-1}\bigg) > \delta\bigg) \\ 
&\leq \PP\bigg( \frac{1}{|V_n|} \#\left\lbrace j\in V_n: \norm{S^j X^n}_{T,\lambda} > \delta \right\rbrace > \delta \bigg) \\
&\leq \PP\bigg( \frac{1}{|V_n|} \sum_{j\in V_n} \norm{S^j X^n}_{T,\lambda} > \delta^2 \bigg)\\
&= \PP\bigg( \frac{1}{|V_n|}\sum_{j\in V_n, k\in \Z^d}\lambda^k \norm{X^{n,(j+k)\modd V_n}}_{T} > \delta^2\bigg),
\end{align*}
since $X^{n,m} = X^{n,m\modd V_n}$ for all $m\in \Z^d$. Now we claim that
\begin{equation}
\sum_{j\in V_n,k\in\Z^d}\lambda^k \norm{X^{n,(j+k)\modd V_n}}_{T} = \sum_{l\in V_n} \norm{X^{n,l}}_T.
\end{equation}
This is because for any $l\in V_n$ and $k\in\Z^d$, there exists a unique $j\in V_n$ such that $(j+k)\modd V_n = l$. Hence the coefficient of $ \norm{X^{n,l}}_T$ on the right is $\sum_{k\in\Z^d}\lambda^k = 1$.
Thus, making use of the previous two results, 
\begin{multline}
\PP\bigg(d^{\lambda,\mathcal{P}}\bigg(\hat{\mu}^n(W^n)\circ(\Psi^n)^{-1},\hat{\mu}^n(W^n)\circ \Psi^{-1}\bigg) > \delta\bigg)\\
\leq \Prob\left( \sum_{j\in V_n}\norm{X^{n,j}}_{T} > |V_n| \delta^2\right).
\end{multline}
Using Lemma \ref{Lemma Bound Difference} we may thus conclude that \eqref{eq to prove Dec} is satisfied, i.e. that
\begin{multline}
\lsup{n}\frac{1}{|V_n|} \log \PP\bigg(d^{\lambda,\mathcal{P}}\bigg(\hat{\mu}^n(W^n)\circ(\Psi^n)^{-1},\hat{\mu}^n(W^n)\circ \Psi^{-1}\bigg) > \delta\bigg) = -\infty.
\end{multline}
\end{proof}
Recall that $\tilde{W}^n \in \T^{\Z^d}$ is the periodic interpolant of $W^n$, i.e. such that $\tilde{W}^{n,j} = W^{n,j\modd V_n}$.
\begin{lemma}\label{Lemma Bound Difference}
For any $\delta > 0$,
\begin{equation}
\lsup{n} \frac{1}{|V_n|}\log\Prob\left( \sum_{j\in V_n}\norm{\Psi(\tilde{W}^n)^j-\Psi^n(\tilde{W}^n)^{j}}_{T} > |V_n| \delta^2\right) = -\infty.
\end{equation}
\end{lemma}
\begin{proof}
Write $Y^n = \Psi(\tilde{W}^n)$ and $Z^n = \Psi^n(\tilde{W}^n)$. Suppose that $\tau,t$ are such that $Y^{n,j}_\tau = Z^{n,j}_\tau$ and $\norm{Y^{n,j} - Z^{n,j}}_T = |Y^{n,j}_t- Z^{n,j}_t|$. We may assume that $Y^{n,j}_s- Z^{n,j}_s$ is of the same sign for all $s\in [\tau,t]$. We then see, making use of \eqref{eq bound 3 start}-\eqref{eq bound 4 start} and Assumption \ref{Assumption NonUniform}, that
\begin{align*}
|Y^{n,j}_t - Z^{n,j}_t| &= \int_\tau^t\bigg[ \mathfrak{b}_s(Y^{n,j}) - \mathfrak{b}_s(Z^{n,j}) + \sum_{k\in\Z^d}\big(\Lambda^k_s(Y^{n,j},Y^{n,j+k})\\&-\Lambda^k_s(Z^{n,j},Z^{n,j+k})\big) + \sum_{k\notin V_n}\Lambda^k_s(Z^{n,j},Z^{n,j+k})\bigg]ds \\
&\leq \int_\tau^t \bigg[C \norm{Y^{n,j}-Z^{n,j}}_s + \sum_{k\in\Z^d} \kappa^k \norm{Y^{n,j+k}-Z^{n,j+k}}_s \\ &+ \sum_{k\notin V_n}\kappa^k\big(1+\norm{Z^{n,j}}_s + \norm{Z^{n,j+k}}_s\big)\bigg] ds\\
&\leq \int_0^t \bigg[C \norm{Y^{n,j}-Z^{n,j}}_s + \sum_{k\in\Z^d} \kappa^k \norm{Y^{n,j+k}-Z^{n,j+k}}_s \\ &+ \sum_{k\notin V_n}\kappa^k\big(\norm{Z^{n,j}}_s + \norm{Z^{n,j+k}}_s\big)\bigg] ds + t\bar{\kappa}_n.
\end{align*}
We may thus conclude that
\begin{multline*}
\norm{Y^{n,j} - Z^{n,j}}_t \leq  \int_0^t \bigg[C \norm{Y^{n,j}-Z^{n,j}}_s + \sum_{k\in\Z^d} \kappa^k \norm{Y^{n,j+k}-Z^{n,j+k}}_s \\ + \sum_{k\notin V_n}\kappa^k\big(\norm{Z^{n,j}}_s + \norm{Z^{n,j+k}}_s\big)\bigg] ds + t\bar{\kappa}_n.
\end{multline*}
Hence summing over $j$, we see that
\begin{align*}
\sum_{j\in V_n}\norm{Y^{n,j}-Z^{n,j}}_t &\leq \int_0^t\bigg[ C\sum_{j\in V_n}\norm{Y^{n,j}-Z^{n,j}}_s + \sum_{j\in V_n,k\in \Z^d}\kappa^k \norm{Y^{n,j+k}-Z^{n,j+k}}_s \\ &+ \sum_{j\in V_n,k\notin V_n}\kappa^k\big(\norm{Z^{n,j}}_s + \norm{Z^{n,j+k}}_s\big)\bigg] ds + t|V_n|\bar{\kappa}_n,\\
&= \int_0^t\bigg[ C\sum_{j\in V_n}\norm{Y^{n,j}-Z^{n,j}}_s + \kappa_*\sum_{j\in V_n} \norm{Y^{n,j}-Z^{n,j}}_s \\ &+ 2\bar{\kappa}_n\sum_{j\in V_n}\norm{Z^{n,j}}_s\bigg] ds + t|V_n|\bar{\kappa}_n,
\end{align*}
where we have used the $V_n$-periodicity of $Y^n$ and $Z^n$, i.e. $Y^{n,j} = Y^{n,j\modd V_n}$ and $Z^{n,j} = Z^{n,j\modd V_n}$ (as noted in Lemma \ref{Lemma Solution Equivalence}). By Gronwall's Inequality,
\begin{equation}
\sum_{j\in V_n}\norm{Y^{n,j}-Z^{n,j}}_T \leq C_2\bar{\kappa}_n\bigg(|V_n|+  \sum_{j\in V_n}\norm{Z^{n,j}}_T\bigg),
\end{equation}
for some constant $C_2$. We may thus infer using Lemma \ref{Lemma Bound Psin} that for some constant $C_3$,
\begin{align*}
\sum_{j\in V_n}\norm{Y^{n,j}-Z^{n,j}}_T \leq \bar{\kappa}_n C_3\bigg( |V_n| +  \sum_{j\in V_n}\norm{W^{n,j}}_T\bigg).
\end{align*}
Thus
\begin{multline}
\Prob\left( \sum_{j\in V_n}\norm{Y^{n,j}-Z^{n,j}}_{T} > |V_n| \delta^2\right) \\ \leq \Prob\left(  \bar{\kappa}_n C_3 \sum_{j\in V_n}\norm{W^{n,j}}_T > \delta^2 |V_n| - \bar{\kappa}_n C_3 |V_n|\right) \\
=  \Prob\left(  \sum_{j\in V_n}\norm{W^{n,j}}_T > |V_n|\bigg(\frac{\delta^2}{\bar{\kappa}_n C_3}- 1\bigg)\right). 
\end{multline}
For any $a > 0$, since $\bar{\kappa}_n \to 0$ as $n \to \infty$, for $n$ sufficiently large,
\begin{multline*}
\Prob\left(  \sum_{j\in V_n}\norm{W^{n,j}}_T > |V_n|\bigg(\frac{\delta^2}{\bar{\kappa}_n C_3}- 1\bigg)\right) \leq \Prob\left(  \sum_{j\in V_n}\norm{W^{n,j}}_T > a |V_n|\right).
\end{multline*}
The lemma now follows through Assumption \ref{Assumption Noise LDP}.
\end{proof}
\section{Large Deviation Principle for a Gaussian Process which is a Martingale Correlated Modulo $V_n$}\label{Section Noise LDP}
In this section we give an example of a model of correlated noise, where the noise is a martingale with Gaussian marginals and the correlations are `modulo $V_n$'. We explained why we expect the noise to be correlated in the introduction. We will prove that the noise satisfies Assumption \ref{Assumption Noise LDP}: in particular, that the laws of the empirical measure $\hat{\mu}^n(W^n)$ satisfy a Large Deviation Principle. This model of the noise is used in the Fitzhugh-Nagumo example of Section \ref{Section Example}. The Large Deviation Principle is stated in the main result of this section: Theorem \ref{Theorem PinW}. It is important to note that the LDP for the law of $\hat{\mu}^n(W^n)$ is relative to the topology on $\mathcal{P}(\T^{\Z^d}_\lambda)$ induced by the norm $\norm{\cdot}_{T,\lambda}$: this is not the standard cylindrical topology on $\T^{\Z^d}$. An LDP for the law of $\hat{\mu}^n(W^n)$ relative to the weak topology on $\mathcal{P}(\T^{\Z^d})$ induced by the cylindrical topology on $\T^{\Z^d}$ is an immediate consequence of Theorem \ref{Theorem PinW} through the Contraction Principle \cite[Theorem 4.2.21]{dembo-zeitouni:97}. We now outline our model in more detail.

The noise $W^n := (W^{n,j}_t)_{j\in V_n,t\in [0,T]}$ is a correlated martingale over $\T^{V_n}$, which we define as follows. The correlation is specified to be `modulo $V_n$', in keeping with the general tenor of this paper. The noise $\big(W^{n,j}\big)_{j\in V_n}$ is taken to be a continuous Gaussian process, i.e. such that for any finite set of times $\big( t_q\big)_{p=1}^M \subset [0,T]$, $\big(W^{n,j}_{t_q}\big)_{j\in V_n,p\in [1,M]}$ have a finite-dimensional Gaussian distribution, and $W^{n,j} \in \mathcal{T}$. Let $(a^j)_{j\in \Z^d} \subset \mathcal{C}([0,T],\R)$ be constants such that 
\begin{equation}
\sum_{j\in\Z^d}\norm{a^j}_T < \infty.
\end{equation}
We stipulate that for $j,k \in V_n$, $0\leq s\leq t \leq T$, 
\begin{align}
\Exp[W^{n,j}_t] =& 0 \text{ and }\\
\Exp\big[ W^{n,j}_s W^{n,k}_t\big] =& \int_0^s a^{(k-j)\modd V_n}(r)dr.\label{eq: W expectation definition}
\end{align}
The fact that a process with the above properties exists (subject to Assumption \ref{Assumption tilde a} below) may be inferred from the proof of Lemma \ref{Lemma Identify Pi n Pi n B}: one could define $W^n$ to be $\big(Z^{n,j}\big)_{j\in V_n}$ in \eqref{eq: Zn defn 1}. $Z^n$ is a moving-average (modulo $V_n$) transformation of $|V_n|$ independent Brownian motions. Define, respectively, the discrete and continuous Fourier Transforms, for $k\in V_n$ and $\theta \in [-\pi,\pi]^d$,
\begin{align}
\tilde{a}^{n,k}(t) &:= \sum_{j\in V_n}\exp \bigg(-\frac{2\pi i\langle j,k\rangle}{2n+1}\bigg)a^j(t)\\
\tilde{a}(t,\theta) &:= \sum_{j\in \Z^d}\exp\big(-i\langle j,\theta\rangle\big)a^j(t).
\end{align}
The following assumptions are needed to guarantee that the process $W^{n,j}_t$ exists and is well-behaved.
\begin{assumption}\label{Assumption tilde a}
Assume that if $j(p) = \pm k(p)$ for all $p\in \lbrace 1,\ldots,d\rbrace$, then $a^{j} = a^k$. It follows from this that $\tilde{a}^{n,k}(t),\tilde{a}(t,\theta) \in \R$ and if $j(p) = \pm k(p)$ for all $p\in \lbrace 1,\ldots,d\rbrace$, then $\tilde{a}^{n,j} = \tilde{a}^{n,k}$. We assume that there exists a constant $\tilde{a}^{max}$ such that for all $n\in\Z^+$, $k\in V_n$ and $t\in [0,T]$, $0 \leq \tilde{a}^{n,k}(t) \leq \tilde{a}^{max}$. We also assume that for all $t\in [0,T]$ and $\theta \in [-\pi,\pi]^d$, $0\leq \tilde{a}(t,\theta) \leq \tilde{a}^{max}$.
\end{assumption}
Define, for $\theta \in [-\pi,\pi]^d$ and $j\in \Z^d$,
\begin{align*}
\tilde{c}(t,\theta) &= \sqrt{\tilde{a}(t,\theta)} \\
c^j(t) &= \frac{1}{(2\pi)^d}\int_{[-\pi,\pi]^d}\exp\big(ij\omega\big)\tilde{c}(t,\omega)d\omega,
\end{align*}
noting that these variables are in $\T$. We assume that 
\begin{equation}\label{eq: c norm absolute}
\sum_{j\in \Z^d}\norm{c^j}_T < \infty.
\end{equation}
We also assume that $\frac{d}{dt}\tilde{c}(t,\theta) \in \T$ exists for all $\theta \in [-\pi,\pi]^d$, and that it has an absolutely convergent Fourier Series, i.e. the following properties are satisfied: for $k\in \Z^d$, $\theta \in [-\pi,\pi]^d$ and $t\in [0,T]$,
\begin{align}
\mathfrak{f}^k &:=\frac{d}{dt}c^k \in \T\label{assumption f absolute convergence 0} \\
\mathfrak{f}^k(t) &= \frac{1}{(2\pi)^d}\int_{[-\pi,\pi]^d}\exp\big(i\langle k,\omega\rangle\big)\frac{d}{dt}\tilde{c}(t,\omega)d\omega\\
\frac{d}{dt}\tilde{c}(t,\theta) &= \sum_{k\in\Z^d}\exp\big(-i\langle k,\theta\rangle\big)\mathfrak{f}^k(t) \text{ and }\\
\sum_{k\in \Z^d}\norm{\mathfrak{f}^k}_T &< \infty.\label{assumption f absolute convergence}
\end{align}
For $k\in V_n$, define $\tilde{c}^{n,k}(t) = \sqrt{\tilde{a}^{n,k}(t)}$. Then define, for $j\in V_n$,
\begin{equation*}
c^{n,j}(t) = \frac{1}{|V_n|}\sum_{k\in V_n}\exp\bigg(\frac{2\pi i\langle j,k\rangle}{2n+1}\bigg)\tilde{c}^{n,k}(t).
\end{equation*}
Define $\mathfrak{c}^{n,k}$ as follows. If $k\in V_n$, then
\begin{equation}
\mathfrak{c}^{n,k}(t) = c^{n,k}(t) - c^k(t),\label{defn cnk 1}
\end{equation}
otherwise if $k\notin V_n$, then
\begin{equation}
\mathfrak{c}^{n,k}(t) = -c^k(t).\label{defn cnk 2}
\end{equation}
Let
\begin{equation}\label{defn etan}
\eta_{n,j}= \norm{\mathfrak{c}^{n,j}}_T + T\norm{\frac{d}{dt}\mathfrak{c}^{n,j}}_T,
\end{equation}
and $\eta_{n,*} = \sum_{j\in\Z^d}\eta_{n,j}$. 
\begin{assumption}\label{Assumption Eta}
We assume \eqref{eq: c norm absolute} and \eqref{assumption f absolute convergence 0}- \eqref{assumption f absolute convergence} hold, and that
\begin{equation*}
\lim_{n\to\infty}\eta_{n,*}= 0.
\end{equation*}
\end{assumption}
Since $\sum_{k\in\Z^d}\big(\norm{c^k}_T + \norm{\mathfrak{f}^k}_T\big) < \infty$,  and recalling that $\kappa^{k} = \kappa^{-k}$ (as noted in \eqref{eq kappa jk}), we may assume that
\begin{equation}\label{eq: ck inequality}
\norm{c^k}_T + \norm{\mathfrak{f}^k}_T \leq \kappa^{k},\kappa^{-k}
\end{equation}
(if necessary we can make the change of definition $\kappa^{k}\to \max\big\lbrace\kappa^{k},\big(\norm{c^k}_T + \norm{\mathfrak{f}^k}_T\big)\big\rbrace$). By Lemma \ref{Lemma Switch Inequality},
\begin{equation}\label{eq: lambda c inequality}
\sum_{j\in\Z^d}\lambda^{j}\norm{c^{m-j}}_T \leq 2\kappa_*\lambda^m \text{ and }\sum_{j\in\Z^d}\lambda^{j}\norm{\mathfrak{f}^{m-j}}_T \leq 2\kappa_*\lambda^m .
\end{equation}
Let $\Pi^n_W$ be the law of $\hat{\mu}^n(W^n)$. The main result of this section is the following.
\begin{theorem}\label{Theorem PinW}
Under Assumptions \ref{Assumption tilde a} and \ref{Assumption Eta}, the laws $(\Pi^n_W)_{n\in\Z^+} \subset \mathcal{P}\big(\mathcal{P}(\T^{\Z^d}_\lambda)\big)$ satisfy a Large Deviation Principle with good rate function.  This LDP is relative to the topology on $\mathcal{P}(\T^{\Z^d}_\lambda)$ generated by the Levy-Prokhorov metric $d^{\lambda,\mathcal{P}}(\cdot,\cdot)$ defined in Section \ref{Subsection Preliminaries}. 
\end{theorem}
Before we prove this theorem, we make some more definitions and prove some introductory results. Define $\Gamma^n,\Gamma: \mathcal{T}^{\Z^d}_\lambda \to \mathcal{T}^{\Z^d}_\lambda$ as follows. For $w\in \T^{\Z^d}_\lambda$, $\Gamma^n(w) := Z^n$ and $\Gamma(w) := Y$, where for $t\in [0,T]$
\begin{align}
Z^{n,j}_t =&\; \sum_{k\in V_n}\left(c^{n,k}(t)w^{j-k}_t - \int_0^t w^{j-k}_s\frac{d}{ds}c^{n,k}(s)ds\right) \text{ for }j\in V_n, \label{eq: Gamma Defn 1}\\
Z^{n,j} =&\; Z^{n,j\modd V_n} \text{ for }j\notin V_n,\\
Y^j_t =&\; \sum_{k\in \Z^d}\left(c^{k}(t)w^{j-k}_t - \int_0^t w^{j-k}_s\frac{d}{ds}c^{k}(s)ds\right) \text{ for }j\in\Z^d.\label{eq: Gamma Defn 3}
\end{align}
Define $\Gamma^n_{\mathcal{P}},\Gamma_{\mathcal{P}}: \mathcal{P}(\mathcal{T}^{\Z^d}_\lambda) \to \mathcal{P}(\mathcal{T}^{\Z^d}_\lambda)$ by
\begin{align*}
\Gamma^n_{\mathcal{P}}(\mu) :=& \mu \circ (\Gamma^n)^{-1} \\
\Gamma_{\mathcal{P}}(\mu) :=& \mu \circ \Gamma^{-1}.
\end{align*}

\begin{lemma}\label{Lemma Gamma Continuity}
The maps $\Gamma,\Gamma^n : \T^{\Z^d}_\lambda \to \T^{\Z^d}_\lambda$, as well as the maps $\Gamma_{\mathcal{P}},\Gamma^n_{\mathcal{P}} : \T^{\Z^d}_\lambda \to \T^{\Z^d}_\lambda$, are well-defined and continuous.
\end{lemma}
\begin{proof}
After one has noted that the support of $\Gamma^n$ lies in $\T^{V_{2n}}$, it is not too difficult to see that it is well-defined and continuous. The existence and continuity of $\Gamma_{\mathcal{P}}$ and $\Gamma^n_{\mathcal{P}}$ follows immediately from that of $\Gamma$ and $\Gamma^n$.

The existence and continuity of the map $\Gamma$ follows from the following consideration. Suppose that for $w,v\in\T^{\Z^d}_\lambda$,
\begin{align*}
Y^j_t =& \sum_{k\in \Z^d}\left(c^{k}(t)w^{j-k}_t - \int_0^t w^{j-k}_s\frac{d}{ds}c^{k}(s)ds\right)\\
X^j_t =& \sum_{k\in \Z^d}\left(c^{k}(t)v^{j-k}_t - \int_0^t v^{j-k}_s\frac{d}{ds}c^{k}(s)ds\right).
\end{align*}
Then, using \eqref{eq: ck inequality}, Jensen's Inequality, and recalling that $\kappa_* = \sum_{j\in \Z^d}\kappa^j$,
\begin{align*}
&\sum_{j\in\Z^d}\lambda^j\norm{Y^j - X^j}^2_t \\ \leq & \sum_{j,k\in \Z^d}\lambda^j\left(\norm{c^{k}}_t\norm{w^{j-k} - v^{j-k}}_t + \int_0^t \norm{w^{j-k} - v^{j-k}}_s\norm{\frac{d}{ds}c^{k}}_s ds\right)^2 \\
\leq&\sum_{j\in\Z^d}\lambda^j\big((1+t)\sum_{k\in\Z^d}\kappa^{-k}\norm{w^{j-k}-v^{j-k}}_t\big)^2\\
\leq & \kappa_*(1+t)^2\sum_{j,k\in \Z^d}\lambda^j\kappa^{-k}\norm{w^{j-k} - v^{j-k}}^2_t \\
\leq & 2\kappa_*^2(1+t)^2 \sum_{j\in\Z^d}\lambda^j\norm{w^j - v^j}^2_t,
\end{align*}
using \eqref{eq: lambda c inequality}. If we take $v^j = 0$ for all $j\in\Z^d$, then we see that $\Gamma(w) \in \T^{\Z^d}_\lambda$ is well-defined. The above identity also demonstrates that $\Gamma$ is Lipschitz.
\end{proof}
Let $(B^j)_{j\in\Z^d}$ be independent $\R$-valued Wiener Processes on $[0,T]$. Let $\Pi^n_{B}$ be the law of $\hat{\mu}^n(B) \in \mathcal{P}\big(\T^{\Z^d}_\lambda\big)$. 
\begin{theorem}\label{Theorem LDP PinB}
The laws $(\Pi^n_B)_{n\in\Z^+} \subset \mathcal{P}\big( \mathcal{P}(\T^{\Z^d}_\lambda)\big)$ satisfy an LDP with good rate function. This LDP is relative to the topology on $\mathcal{P}(\T^{\Z^d}_\lambda)$ generated by the Levy-Prokhorov Metric $d^{\lambda,\mathcal{P}}(\cdot,\cdot)$ defined in Section \ref{Subsection Preliminaries}.
\end{theorem}
This theorem is proved in Section \ref{Section Proof of Last Theorem} below.
\begin{lemma}\label{Lemma Identify Pi n Pi n B}
\begin{equation*}
\Pi^n_W = \Pi^n_B\circ(\Gamma_{\mathcal{P}}^n)^{-1}.
\end{equation*}
\end{lemma}
\begin{proof}
We fix $n$ throughout this proof. Let $\tilde{B} \in \T^{\Z^d}_\lambda$ be the $V_n$-periodic interpolant of $(B^j)_{j\in V_n}$, i.e. such that for all $k\in \Z^d$, $\tilde{B}^k := B^{k\modd V_n}$. 

\vspace{0.2cm}
\textit{Claim: The law of $\frac{1}{|V_n|}\sum_{j\in V_n}\delta_{S^j (\Gamma^n(\tilde{B}))}$ is $\Pi^n_B\circ(\Gamma_{\mathcal{P}}^n)^{-1}$.}
\vspace{0.2cm}

Write $Z^{n} = \Gamma^n(\tilde{B})$ and notice that for all $j\in \Z^d$
\begin{equation}\label{eq: Zn defn 1}
Z^{n,j}_t = \sum_{k\in V_n}\left(c^{n,k}(t)B^{(j-k)\modd V_n}_t - \int_0^t B^{(j-k)\modd V_n}_s\frac{d}{ds}c^{n,k}(s)ds\right).
\end{equation}
We thus observe that for $j\in V_n$, 
\begin{equation}\label{eq B shift invariance}
S^j \Gamma^n\big(\tilde{B}\big) = \Gamma^n\big(S^j \tilde{B}\big). 
\end{equation}
The claim follows from this observation.

Since $Z^n$ and $W^n$ are $V_n$-periodic, it thus suffices to show that $\big(Z^{n,j}\big)_{j\in V_n}$ has the same law as $\big(W^{n,j}\big)_{j\in V_n}$. Now $\big(Z^{n,j}\big)_{j\in V_n}$ and $W^n$ are both Gaussian, and therefore we merely need to show that the mean and variance are the same. Both of these processes have zero mean. An application of Ito's Lemma yields that $\PP$-almost-surely,
\begin{equation}\label{eq: bound Wjt}
Z^{n,j}_t = \sum_{k\in V_n} \int_0^t c^{n,k}(s) dB^{(j-k)\modd V_n}_s.
\end{equation}
We observe that the covariances are invariant under shifts modulo $V_n$, i.e. for $j,k,l \in V_n$ and $s,t \in [0,T]$,
\begin{align*}
\mathbb{E}\left[ Z^{n,j}_s Z^{n,k}_t\right] &= \mathbb{E}\left[ Z^{n,(j+l)\modd V_n}_s Z^{n,(k+l)\modd V_n}_t\right] \\
\mathbb{E}\left[ W^{n,j}_s W^{n,k}_t\right] &= \mathbb{E}\left[ W^{n,(j+l)\modd V_n}_s W^{n,(k+l)\modd V_n}_t\right]. 
\end{align*}
Thus, it suffices for us to show that for all $m\in V_n$ and $t,u\in [0,T]$, 
\begin{align*}
\mathbb{E}\left[ W^{n,0}_t W^{n,m}_u\right]  = \mathbb{E}\left[ Z^{n,0}_t Z^{n,m}_u\right]. 
\end{align*}
We note also that $(Z^{n,j})_{j\in V_n}$ and $W^n$ are both martingales. This means that, we only need to verify the above expression in the case that $t=u$. In sum, our remaining task is to prove that for all $m\in V_n$ and $t\in [0,T]$,
\begin{align*}
\mathbb{E}\left[ W^{n,0}_t W^{n,m}_t\right] = \sum_{k,l\in V_n}\mathbb{E}\left[  \int_0^t c^{n,k}(s) dB^{-k}_s \int_0^t c^{n,l}(r) dB^{(m-l)\modd V_n}_r\right].
\end{align*}
Now $\mathbb{E}\left[  \int_0^t c^{n,k}(s) dB^{-k}_s \int_0^t c^{n,l}(r) dB^{m-l}_r\right]$ is nonzero if and only if $-k \modd V_n = (m-l) \modd V_n$. Furthermore $-k \modd V_n = (m-l) \modd V_n$ if and only if $l \modd V_n = (m+k)\modd V_n$. We thus see, using the Ito Isometry, that
\begin{align*}
\mathbb{E}\left[ Z^{n,0}_t Z^{n,m}_t\right] =& \sum_{k\in V_n} \int_0^t c^{n,k}(s)c^{n,(k+m)\modd V_n}(s)ds\\
=&\int_0^t  \sum_{k\in V_n}c^{n,-k}(s)c^{n,(k+m)\modd V_n}(s)ds,
\end{align*}
since $c^{n,k}(s) = c^{n,-k}(s)$. It follows from the convolution formula for the discrete Fourier Transform that $\sum_{k\in V_n}c^{n,-k}(s)c^{n,(k+m)\modd V_n}(s)$ is the $m^{th}$ discrete Fourier coefficient of $\big((\tilde{c}^{n,k}(s))^2\big)_{k\in V_n}$. That is, $\sum_{k\in V_n}c^{n,-k}(s)c^{n,(k+m)\modd V_n}(s) = a^m(s)$. In light of \eqref{eq: W expectation definition}, we thus see that
\begin{equation*}
\mathbb{E}\left[ Z^{n,0}_t Z^{n,m}_t\right] = \mathbb{E}\left[ W^{n,0}_t W^{n,m}_t\right],
\end{equation*}
as required.
\end{proof}
The following proof uses some ideas from \cite[Theorem 4.9]{baxter-jain:93}.
\begin{proof}[Proof of Theorem \ref{Theorem PinW}]
From Lemma \ref{Lemma Identify Pi n Pi n B}, $\Pi^n_W$ is the law of $\hat{\mu}^n(B)\circ(\Gamma^n)^{-1}$. Therefore from \cite[Corollary 4.2.21]{dembo-zeitouni:97}, and the fact that the maps $\mu \to \mu\circ(\Gamma^n)^{-1}$ and $\mu \to \mu\circ\Gamma^{-1}$ are continuous on $\mathcal{P}(\T_\lambda^{\Z^d})$ (thanks to Lemma \ref{Lemma Gamma Continuity}), it suffices to prove that for any $\delta > 0$,
\begin{equation}\label{eq to prove Dec}
\lim_{n\to\infty}\frac{1}{|V_n|}\log\PP\bigg(d^{\lambda,\mathcal{P}}\bigg(\hat{\mu}^n(B)\circ(\Gamma^n)^{-1},\hat{\mu}^n(B)\circ \Gamma^{-1}\bigg) > \delta\bigg) = -\infty.
\end{equation}
Let $\tilde{B} \in \T^{\Z^d}_\lambda$ be the $V_n$-periodic interpolation of $(B^j)_{j\in V_n}$ - i.e. such that $\tilde{B}^k := B^{k\modd V_n}$ for all $k\in \Z^d$. Note that $\tilde{B}$ clearly depends on $n$. Let $X^n = \Gamma(\tilde{B}) - \Gamma^n(\tilde{B})$.
Then, very similarly to the proof of Theorem \ref{Theorem Main LDP},
\begin{align*}
d^{\lambda,\mathcal{P}}\bigg(\hat{\mu}^n(B)\circ(\Gamma^n)^{-1},\hat{\mu}^n(B)\circ \Gamma^{-1}\bigg) \leq \max\left\lbrace \delta, \frac{1}{|V_n|} \#\lbrace j\in V_n: \norm{S^j X^n}_{T,\lambda} > \delta\rbrace \right\rbrace \\
=  \max\left\lbrace \delta,\frac{1}{|V_n|} \#\left\lbrace j\in V_n: \norm{S^j X^n}_{T,\lambda}^2 > \delta^2\right\rbrace \right\rbrace.
\end{align*}
Hence
\begin{align*}
\PP\bigg(d^{\lambda,\mathcal{P}}\bigg(\hat{\mu}^n(B)\circ(\Gamma^n)^{-1}&,\hat{\mu}^n(B)\circ \Gamma^{-1}\bigg) > \delta\bigg) \\ 
&\leq \PP\bigg( \frac{1}{|V_n|} \#\lbrace j\in V_n: \norm{S^j X^n}_{T,\lambda}^2 > \delta^2\rbrace > \delta \bigg) \\
&\leq \PP\bigg( \frac{1}{|V_n|} \sum_{j\in V_n} \norm{S^j X^n}_{T,\lambda}^2 > \delta^3\bigg)\\
&= \PP\bigg( \frac{1}{|V_n|}\sum_{j\in V_n, k\in \Z^d}\lambda^k \norm{X^{n,(j+k)\modd V_n}}^2_{T} > \delta^3\bigg),
\end{align*}
since $X^{n,m} = X^{n,m\modd V_n}$ for all $m\in \Z^d$. Now we claim that
\begin{equation}
\sum_{j\in V_n,k\in\Z^d}\lambda^k \norm{X^{n,(j+k)\modd V_n}}^2_{T} = \sum_{l\in V_n} \norm{X^{n,l}}^2_T.
\end{equation}
This is because for any $l\in V_n$ and $k\in\Z^d$, there exists a unique $j\in V_n$ such that $(j+k)\modd V_n = l$. Hence the coefficient of $ \norm{X^{n,l}}^2_T$ on the right is $\sum_{k\in\Z^d}\lambda^k = 1$.
Thus, making use of the previous two results, 
\begin{align}
\PP\bigg(&d^{\lambda,\mathcal{P}}\bigg(\hat{\mu}^n(B)\circ(\Gamma^n)^{-1},\hat{\mu}^n(B)\circ \Gamma^{-1}\bigg) > \delta\bigg)\nonumber\\
&\leq \Prob\left( \sum_{j\in V_n}\norm{X^{n,j}}^2_{T} > |V_n| \delta^3\right) \nonumber\\
&\leq \exp\left( -b |V_n|\delta^3\right)\Exp\bigg[\exp\bigg(b\sum_{j\in V_n} \norm{X^{n,j}}^2_{T}\bigg)\bigg]\label{eq Dec bnd 1}
\end{align}
for some $b>0$, through Chebyshev's Inequality. 

We thus see that, after noting the definition of $\Gamma^n$ and $\Gamma$ in \eqref{eq: Gamma Defn 1}-\eqref{eq: Gamma Defn 3}, and the definition $\mathfrak{c}^{n,k}$ in \eqref{defn cnk 1}-\eqref{defn cnk 2} that for any $j\in \Z^d$,
\begin{align*}
X^{n,j}_t =& \sum_{k\in \Z^d}\left(\mathfrak{c}^{n,k}(t)B^{(j-k)\modd V_n}_t - \int_0^t B^{(j-k)\modd V_n}_s\frac{d}{ds}\mathfrak{c}^{n,k}(s)ds\right),\\
\norm{X^{n,j}}_T \leq & \sum_{k \in \Z^d}\eta_{n,k} \norm{B^{(j-k)\modd V_n}}_T\\
\norm{X^{n,j}}_T^2 \leq & \eta_{n,*}\sum_{k\in \Z^d}\eta_{n,k}\norm{B^{(j-k)\modd V_n}}_T^2,
\end{align*}
where this last step follows by the Cauchy-Schwarz Inequality, $\eta_{n,k}$ is defined in \eqref{defn etan} and $\eta_{n,*} = \sum_{m \in \Z^d}\eta_{n,m}$. Thus, for some positive constant $b$, assuming for the moment that the following integrals are well-defined, we find that
\begin{equation*}
\Exp\bigg[\exp\big(b\sum_{j\in V_n} \norm{X^{n,j}}^2_{T}\big)\bigg] \leq \Exp\bigg[\exp\bigg(b\eta_{n,*}\sum_{j\in V_n,k\in\Z^d}\eta_{n,k}\norm{B^{(j-k)\modd V_n}}_T^2\bigg)\bigg].
\end{equation*}
We claim that 
\begin{equation*}
\sum_{j\in V_n,k\in\Z^d}\eta_{n,k}\norm{B^{(j-k)\modd V_n}}_T^2 = \sum_{k\in\Z^d}\eta_{n,k} \sum_{l\in V_n}\norm{B^l}_T^2 = \eta_{n,*}\sum_{l\in V_n}\norm{B^l}_T^2 .
\end{equation*}
This is because for each $k\in \Z^d$ and $l\in V_n$, there is a unique $j\in V_n$ such that $(j-k)\modd V_n = l$. This is given by $j=(k+l)\modd V_n$. The above two results imply that\begin{equation*}
\Exp\bigg[\exp\big(b\sum_{j\in V_n} \norm{X^{n,j}}^2_{T}\big)\bigg]  \leq \Exp\bigg[\exp\bigg(b\eta_{n,*}^2\sum_{m\in V_n}\norm{B^{m}}_T^2\bigg)\bigg].
\end{equation*}
Observe that the coefficient of each $\norm{B^m}^2_T$ is less than or equal to $b\eta_{n,*}^2$, which goes to zero as $n\to \infty$ by Assumption \ref{Assumption Eta}. Hence for $n$ large enough, by Lemma \ref{Lemma bound B} the previous expectation is finite, and satisfies the bound
\begin{equation}
\Exp\bigg[\exp\bigg(b\eta_{n,*}^2\sum_{m\in V_n}\norm{B^{m}}_T^2\bigg)\bigg]\leq \exp\big(|V_n|\grave{C} b\eta_{n,*}^2\big).\label{eq Dec bnd 2}
\end{equation}
Combining \eqref{eq Dec bnd 1} and \eqref{eq Dec bnd 2}, since $\eta_{n,*}\to 0$ as $n\to \infty$, we see that
\begin{equation*}
\lim_{n\to\infty}\frac{1}{|V_n|}\log\PP\bigg(d^{\lambda,\mathcal{P}}\bigg(\hat{\mu}^n(B)\circ(\Gamma^n)^{-1},\hat{\mu}^n(B)\circ \Gamma^{-1}\bigg) > \delta\bigg) 
\leq -b\delta^3.
\end{equation*}
We take $b\to \infty$ to obtain \eqref{eq to prove Dec}. 
\end{proof}
\begin{lemma}\label{Lemma bound B}
There exists a constant $\grave{C}$ such that for all $a\in [0,\frac{1}{4T}]$,
\[
\Exp\left[\exp\big(a\norm{B^j}_T^2\big)\right] \leq 1 + \grave{C} a \leq \exp\big(\grave{C} a\big).
\]
\end{lemma}
\begin{proof}
Now the reflection principle \cite[Section 2.6]{karatzas-shreve:91} dictates that, for any $b\geq 0$, $\Prob(\sup_{t\in [0,T]}B^j_t \geq b) = \int_b^\infty 2(2\pi T)^{-\frac{1}{2}}\exp\big(-\frac{x^2}{2T}\big)dx$. Hence, since $(-B_t)_{t\in [0,T]}$ has the same law as $(B_t)_{t\in [0,T]}$,
\begin{equation}
\Prob(\norm{B^j}_T \geq b) \leq \int_b^\infty 4(2\pi T)^{-\frac{1}{2}}\exp\bigg(-\frac{x^2}{2T}\bigg)dx.
\end{equation}
We obtain an upper bound for the expectation of $\exp\big(a\norm{B^j}_T^2)$ by assuming that the above density assumes its maximal value for all $b \geq L$, where
$L > 0$ is chosen to ensure that the integral of the density upper bound is one. That is, we find that
\begin{align}
\Exp\left[\exp\big(a\norm{B^j}_T^2\big)\right] \leq \int_L^\infty 4(2\pi T)^{-\frac{1}{2}}\exp\bigg(-\frac{x^2}{2T} + ax^2\bigg)dx.\label{eq exp tight 13}
\end{align}
The limit $L$ is defined such that
\[
\int_L^\infty 4(2\pi T)^{-\frac{1}{2}}\exp\bigg(-\frac{x^2}{2T}\bigg)dx = 1.
\]
Note that the above is equivalent to requiring that $\erf\big(L (2T)^{-\frac{1}{2}}\big) = \frac{1}{2}$, where $\erf(x) := 2\pi^{-\frac{1}{2}}\int_0^x \exp(-t^2)dt$. Now through a change of variable, we see that if $\bar{L} = L\sqrt{(2T)^{-1} - a}$, then
\begin{align}
\int_L^\infty 4(2\pi T)^{-\frac{1}{2}}\exp\big(-\frac{x^2}{2T} + ax^2\big)dx &= (1-2aT)^{-\frac{1}{2}}\int_{\bar{L}}^\infty \frac{2}{\sqrt{\pi}}\exp(-y^2)dy \\
&= h(a) ,
\end{align}
where $h(a) := (1-2aT)^{-\frac{1}{2}}\big( 1 - \erf(\bar{L})\big)$ (note the dependence of $\bar{L}$ on $a$). Now for $a\in [0,\frac{1}{4T}]$, the function $h(a)$ is differentiable, with $|h'(a)|$ uniformly bounded on this interval. Hence by the Mean Value Theorem, since $h(0) = 1$, there exists a constant $\grave{C}$ such that for all $a\in [0,\frac{1}{4T}]$,
\begin{equation}
\int_L^\infty 4(2\pi T)^{-\frac{1}{2}}\exp\bigg(-\frac{x^2}{2T} + ax^2\bigg)dx \leq 1 + \grave{C}a.
\end{equation}
We may finally note that $1+\grave{C}a \leq \exp\big(\grave{C}a\big)$ as a consequence of Taylor's Theorem. This gives us the lemma.
\end{proof}
The following lemma is needed in order that Assumption \ref{Assumption Noise LDP} is satisfied.
\begin{lemma}\label{exp bound Wn}
\begin{equation}\label{eq: a limit}
\lim_{a\to\infty}\lsup{n}\frac{1}{|V_n|}\log\Prob\left(\sum_{j\in V_n}\norm{W^{n,j}}_T > a |V_n|\right) = -\infty.
\end{equation}
\end{lemma}
\begin{proof}
We noted in the proof of Lemma \ref{Lemma Identify Pi n Pi n B} that $W^n$ has the same law as $Z^n$ in \eqref{eq: Zn defn 1}. For the purpose of taking the expectation in the lemma, we may therefore assume that for independent Brownian motions $(B^j)_{j\in V_n}$,
\begin{equation}\label{eq Wnj Ito}
W^{n,j}_t = \sum_{k\in V_n}\left(c^{n,k}(t)B^{(j-k)\modd V_n}_t - \int_0^t B^{(j-k)\modd V_n}_s\frac{d}{ds}c^{n,k}(s)ds\right).
\end{equation}
We first establish that there exists a constant $\breve{C}$ such that for all $n\in \Z^+$,
\begin{equation}\label{eq claim 1 to establish}
\Exp\bigg[ \exp\big(\breve{C}\sum_{j\in V_n}\norm{W^{n,j}}_T^2\big)\bigg] \leq \exp\left( \grave{C} |V_n|\right),
\end{equation}
where $\grave{C}$ is the constant in Lemma \ref{Lemma bound B}. Let 
\begin{align*}
\upsilon^{n,j} :=& \norm{c^{n,j}}_T + T\norm{\frac{d}{dt}c^{n,j}}_T,\\
K =& \sup_{n\geq 1}\sum_{j\in V_n}\upsilon^{n,j}.
\end{align*}
$K<\infty$ thanks to \eqref{eq: c norm absolute}, \eqref{assumption f absolute convergence} and Assumption \ref{Assumption Eta}. We observe from \eqref{eq Wnj Ito} that
\begin{align*}
\norm{W^{n,j}}_T \leq & \sum_{k\in V_n}\upsilon^{n,k} \norm{B^{(j-k)\modd V_n}}_T \\
\norm{W^{n,j}}^2_T \leq & K \sum_{k\in V_n}\upsilon^{n,k} \norm{B^{(j-k)\modd V_n}}_T^2,
\end{align*}
using Jensen's Inequality. Thus
\begin{align*}
\exp\bigg(\frac{1}{4TK^2}\sum_{j\in V_n}\norm{W^{n,j}}_T^2\bigg) &\leq \exp\bigg(\frac{1}{4TK}\sum_{j,k\in V_n}\upsilon^{n,k}\norm{B^{(j-k)\modd V_n}}_T^2\bigg)\\
&\leq \exp\bigg(\frac{1}{4T}\sum_{j\in V_n}\norm{B^j}_T^2\bigg).
\end{align*}
Hence, making use of Lemma \ref{Lemma bound B},
\begin{equation*}
\Exp\bigg[\exp\bigg(\frac{1}{4TK^2}\sum_{j\in V_n}\norm{W^{n,j}}_T^2\bigg)\bigg] \leq \exp\left( \grave{C} |V_n|\right).
\end{equation*}
We have thus established \eqref{eq claim 1 to establish}.

We observe from \eqref{eq claim 1 to establish} that for all $n\in\Z^+$,
\begin{align*}
\Prob\bigg(\sum_{j\in V_n}\norm{W^{n,j}}_T - a |V_n| > 0\bigg) &\leq \exp\big(-a\breve{C}|V_n| \big)\Exp\bigg[\exp\big(\breve{C}\sum_{j\in V_n}\norm{W^{n,j}}_T^2\big)\bigg]\\
&\leq \exp\bigg(|V_n| \big(\grave{C}-a\breve{C} \big) \bigg),
\end{align*}
from which the lemma follows.
\end{proof}

\subsection{Proof of Theorem \ref{Theorem LDP PinB}}\label{Section Proof of Last Theorem}
Recall that $(B^j)_{j\in\Z^d}$ are independent Brownian Motions on $[0,T]$ and $\hat{\mu}^n(B)$ is the empirical measure. We denote the weak topology on $\mathcal{P}(\T^{\Z^d}_\lambda)$ (generated by the norm $\norm{\cdot}_\lambda$ on $\T^{\Z^d}_\lambda$) by $\tau_\lambda$. We recall the cylindrical topology on $\T^{\Z^d}$, which is generated by sets $O \subset \T^{\Z^d}$ such that $\pi^{V_m}O$ is open in $\T^{V_m}$ for some $m\in \Z^+$. We let $\tau_W$ be the weak topology on $\mathcal{P}(\T^{\Z^d})$ generated by the cylindrical topology on $\T^{\Z^d}$. It may be seen that the embedding $\T^{\Z^d}_\lambda \hookrightarrow \T^{\Z^d}$ is continuous and injective, and induces a continuous and injective embedding $\mathcal{P}(\T^{\Z^d}_\lambda) \hookrightarrow \mathcal{P}(\T^{\Z^d})$. Note that these embeddings are not necessarily closed. In a slight abuse of notation we identify $\Pi^n_B$ with its image law under this embedding, so that in other words we may also consider $\Pi^n_B$ to be in  $\mathcal{P}(\mathcal{P}(\T^{\Z^d}))$.

The following result is essentially already known.
\begin{theorem}\label{Theorem PinB LDP 1}
$(\Pi^n_B)_{n\in\Z^+}$ satisfy an LDP on $\mathcal{P}(\T^{\Z^d})$ with good rate function (with respect to the cylindrical topology $\tau_W$).
\end{theorem}
\begin{proof}
Recall that $\Pi^n_B$ is the law of the periodic empirical measure $\hat{\mu}^n(B) = \frac{1}{|V_n|}\sum_{j\in V_n}\delta_{S^j \tilde{B}} \in \mathcal{P}(\T^{\Z^d})$, where $\tilde{B}^j :=B^{j\modd V_n}$. Let $\bar{\Pi}^n_B$ be the law of $\bar{\mu}^n(B) := \frac{1}{|V_n|}\sum_{j\in V_n}\delta_{S^j B}$, where $B = (B^j)_{j\in\Z^d} \in \T^{\Z^d}$. Notice that $\bar{\mu}^n(B)$ is not periodically interpolated and not (in general) invariant under shifts of the lattice. Since the $(B^j)_{j\in\Z^d}$ are independent, it is a consequence of \cite[Theorem 1.3]{deuschel-stroock-etal:91} that $(\bar{\Pi}^n_B)_{n\in\Z^+}$ satisfy an LDP with good rate function (with respect to the cylindrical topology $\tau_W$).

The equivalence of the LDPs (relative to the topology $\tau_W$) for $\bar{\mu}^n(B)$ and $\hat{\mu}^n(B)$ when $d=1$ is already known (see for instance \cite[Exercise 6.15]{rassoul2015course}). The proof easily generalises to the case $d\neq 1$.
\end{proof}
Since, the embedding $\mathcal{P}(\T^{\Z^d}_\lambda)\hookrightarrow \mathcal{P}(\T^{\Z^d})$ is continuous and injective, to prove Theorem \ref{Theorem LDP PinB} it suffices in light of Theorem \ref{Theorem PinB LDP 1} and \cite[Theorem 4.2.4]{dembo-zeitouni:97} that we prove that $(\Pi^n_B)_{n\in\Z^+}$ are exponentially tight relative to the topology $\tau_\lambda$ on $\mathcal{P}(\T^{\Z^d}_\lambda)$. This is stated in the following proposition.

\begin{proposition}\label{Proposition Exp Tightness Pi W}
For every $\alpha > 0$, there exists a compact (relative to the topology $\tau_\lambda$) set $\bar{\mathcal{K}}\subset \mathcal{P}(\T_\lambda^{\Z^d})$ such that
\begin{equation}\label{eq exp tight compact 2}
\lsup{n}\frac{1}{|V_n|}\log\Pi^n_B(\bar{\mathcal{K}}^c) \leq - \alpha.
\end{equation}
\end{proposition}
The rest of this section is directed towards the proof of Proposition \ref{Proposition Exp Tightness Pi W}. We introduce the following alternative set of weights $(\beta_m)_{m\in\Z^+}$, $\beta_m := (\beta_m^j)_{j\in\Z^d}$, which put less and less weight on cubes in $\Z^d$ centred at $0$. For $m\in \Z^+$, let $(\beta^j_m)_{j\in\Z^d}$ be such that $\beta^j_m > 0$ for all $j\in\Z^d$, and for some sequence $\xi(m) \subset \R^+$ with $\xi(m)\to\infty$ as $m\to \infty$,
\begin{align}
\beta^j_m &= \xi(m)\lambda^j \text{ for all }j\notin V_m,\label{defn:beta1} \\
\beta^j_m &=|V_m|^{-1} \bigg( 1 -  \xi(m)\sum_{k\notin V_m}\lambda^k\bigg) \text{ for all }j\in V_m,\label{defn:beta2} \\
\sum_{j\in\Z^d}\beta^j_m &= 1.\label{defn:beta3}
\end{align}
Note that \eqref{defn:beta3} follows directly from \eqref{defn:beta1}-\eqref{defn:beta2}. We note that it is possible to find the required sequence $\xi(m)$ for the following reason. For any $M \in \Z^+$, choose $k,l\in \Z^+$, $k \leq l-1$ such that $\sum_{j\notin V_{k}}\lambda^j \leq \frac{1}{2M}$ and $\sum_{j\notin V_{l}}\lambda^j \leq \frac{1}{2(M+1)}$. We may then stipulate that for all $p \in [k,l-1]$, $\xi(p) = M$, and $\xi(l) = M+1$. This process may be continued iteratively as $M\to \infty$.

For $r > 0$, define 
\begin{equation}\label{defn Yrm}
\mathcal{Y}^r_m = \lbrace X \in \T^{\Z^d}_\lambda | \norm{X}^2_{T,\beta_m} \leq r\rbrace,
\end{equation}
where $\norm{X}^2_{T,\beta_m} := \sum_{j\in \Z^d} \beta^j_m \norm{X^j}^2_T$. This set is closed in $\T^{\Z^d}_\lambda$ because the norms $\norm{\cdot}_{\lambda}$ and $\norm{\cdot}_{\beta_m}$ are equivalent - i.e. for each $m$ there must exist constants $\underline{C}_m,\bar{C}_m$ such that $\underline{C}_m\norm{\cdot}_{T,\lambda} \leq \norm{\cdot}_{T,\beta_m} \leq \bar{C}_m\norm{\cdot}_{T,\lambda}$. Define
\[
\mathcal{Y}^r = \cap_{m\in \Z^+}\mathcal{Y}^r_m.
\]
Being the intersection of closed sets, $\mathcal{Y}^r$ is also closed (in the topology of $\T^{\Z^d}_\lambda)$. Let $\delta := (\delta_i)_{i\in\Z^+}$ and $\gamma := (\gamma_i)_{i\in\Z^+}$, be such that $\delta_i \to 0$, $\delta_i < 1/2$ for all $i\in\Z^+$ and $\gamma_i \to \infty$.  Let 
\begin{equation}\label{defn c delta gamma}
 \mathcal{C}(\delta,\gamma)  = \big\lbrace \mu \in  \mathcal{P}(\T^{\Z^d}_\lambda):\mu\big( \mathcal{Y}^{\gamma_i} \big) \geq 1 - \delta_i \text{ for all }i\in\Z^+\big\rbrace.
\end{equation}
It is straightforward to show that $\mathcal{B}(\T^{\Z^d}_\lambda) \subset \mathcal{B}(\T^{\Z^d})$ (the latter is the $\sigma$-algebra generated by the cylindrical topology). Hence any $\mu \in \mathcal{P}(\T^{\Z^d})$ such that $\mu(\T^{\Z^d}_\lambda) = 1$ may be considered to be in $\mathcal{P}(\T^{\Z^d}_\lambda)$.
\begin{lemma}\label{Lemma Compact B K}
Suppose that $\mathcal{U}\subset \mathcal{P}(\T^{\Z^d})$ is compact in $\tau_W$ and that $\mathcal{U}\cap\mathcal{C}(\delta,\gamma)$ is nonempty. Then there exists a set $\mathcal{V}$ such that $\mathcal{U}\cap\mathcal{C}(\delta,\gamma)\subseteq\mathcal{V} \subseteq \mathcal{C}(\delta,\gamma) \subset \mathcal{P}(\T^{\Z^d}_\lambda)$, and $\mathcal{V}$ is compact in the topology $\tau_\lambda$. 
\end{lemma}
\begin{proof}
By Prokhorov's Theorem, since $\mathcal{U}$ is compact in $\tau_W$, there must exist compact subsets $(K_i)_{i \in \Z^+}$, $K_i \subset \T^{\Z^d}$ such that for all $\mu \in \mathcal{U}$, $\mu(K_i^c) < \delta_i$. Define 
\[
\tilde{K}_i = K_i \cap \mathcal{Y}^{\gamma_i} \subset \T^{\Z^d}_\lambda.
\]
Since $\mathcal{U}\cap\mathcal{C}(\delta,\gamma)$ is nonempty by assumption, for some $\mu \in \mathcal{U}\cap\mathcal{C}(\delta,\gamma)$, $\mu(K_i) \geq 1-\delta_i$ and $\mu\big(\mathcal{Y}^{\gamma_i}\big) \geq 1 - \delta_i$. Since $\delta_i < 1/2$, this means that $\tilde{K}_i$ is nonempty.

Define $\mathcal{V} = \lbrace \mu \in \mathcal{P}(\T^{\Z^d}_\lambda): \mu(\tilde{K}_i) \geq 1- 2\delta_i \text{ for all }i\in\Z^+\rbrace$. It follows from the definitions that $\mathcal{U}\cap\mathcal{C}(\delta,\gamma)\subseteq \mathcal{V}$. We first prove that $\mathcal{V}$ is tight, and then that it is closed. It follows from these two facts (in light of Prokhorov's Theorem) that $\mathcal{V}$ is compact.

\textit{Step 1: $\mathcal{V}$ is tight.}

It suffices by Prokhorov's Theorem for us to prove that $\tilde{K}_i$ is compact in the topology of $\T^{\Z^d}_\lambda$.
Let $\lbrace X_{k}\rbrace_{k=1}^\infty$ be a sequence in $\tilde{K}_i$. Fix $\epsilon >0$, and choose $m$ large enough that $\gamma_i / \xi(m) \leq \epsilon / 8$. Then for all $k \in \Z^+$, since $X_k \in \mathcal{Y}^{\gamma_i}_m$
\begin{align*}
\sum_{j\notin V_m}\lambda^j \norm{X^j_k}^2_T &= \frac{1}{\xi(m)}\sum_{j\notin V_m}\beta^j_m\norm{X^j_k}^2_T \\ &\leq \frac{1}{\xi(m)} \norm{X_k}^2_{T,\beta_m}  \\ &\leq \frac{\gamma_i}{\xi(m)} \leq \frac{\epsilon}{8}.
\end{align*}
Since $K_i$ is compact in the cylindrical topology of $\T^{\Z^d}$, we can choose a subsequence $\big(\tilde{X}_{k}\big)_{k\in\Z^+}$ such that 
\[
\sup_{k_1,k_2\in\Z^+, j\in V_m}\norm{\tilde{X}^j_{k_1} - \tilde{X}^j_{k_2}}_T^2 \leq \frac{\epsilon}{2}.
\]
It follows that
\begin{align*}
\sup_{k_1,k_2\in \Z^+}\norm{\tilde{X}_{k_1}-\tilde{X}_{k_2}}_{T,\lambda}^2 &\leq \sum_{j\in V_m}\lambda^j\sup_{k_1,k_2\in\Z^+}\norm{\tilde{X}^j_{k_1}-\tilde{X}^j_{k_2}}^2_T +\sum_{j\notin V_m}\lambda^j \norm{\tilde{X}^j_{k_1}-\tilde{X}^j_{k_2}}^2_T \\
&\leq \frac{\epsilon}{2}\sum_{j\in V_m}\lambda^j+2\sum_{j\notin V_m}\lambda^j \big(\norm{\tilde{X}^j_{k_1}}_T^2+\norm{\tilde{X}^j_{k_2}}^2_T\big)\\
&\leq \frac{\epsilon}{2} + \frac{\epsilon}{2} = \epsilon.
\end{align*}
We may then repeat this process, obtaining a subsequence $(\bar{X}_j)_{j \in \Z^+}$ of $(\tilde{X}_k)_{k\in\Z^+}$, such that  
\[
\sup_{j_1,j_2\in \Z^+}\norm{\bar{X}_{j_1}-\bar{X}_{j_2}}^2_{T,\lambda} \leq \frac{\epsilon}{2}. 
\]
In this way we obtain a Cauchy sequence $\tilde{X}_1,\bar{X}_1,\ldots $, which must converge to a limit in $\T^{\Z^d}_\lambda$. 

It therefore remains for us to prove that $\tilde{K}_i$ is closed in the topology of $\mathcal{T}^{\Z^d}_\lambda$. Since $\mathcal{Y}^{\gamma_i}$ is closed in this topology, it suffices for us to prove that $K_i \cap \T^{\Z^d}_\lambda$ is closed in the topology of $\T^{\Z^d}_\lambda$. Now let $L_m = \lbrace x\in \T^{\Z^d}: \pi^{V_m}x \in \pi^{V_m}K_i\rbrace$. The compactness of $K_i$ means that $L_m$ must be closed in the cylindrical topology on $\T^{\Z^d}$. In turn, it is not too difficult to see that $L_m \cap \T^{\Z^d}_\lambda$ is closed in $\T^{\Z^d}_\lambda$. Therefore $K_i\cap\T^{\Z^d}_\lambda = \cap_{m\in\Z^+}\big(L_m\cap\T^{\Z^d}_\lambda\big)$, being the infinite intersection of closed sets in $\T^{\Z^d}_\lambda$, is closed. 

\textit{Step 2: $\mathcal{V}$ is closed.}

We see from the definition that $\mathcal{V}$ is the intersection of sets of the form $\mathcal{V}_i := \lbrace \mu\in\mathcal{P}(\T^{\Z^d}_\lambda):\mu(\tilde{K}_i) \geq 1- 2\delta_i\rbrace$. Since $\tilde{K}_i$ is closed in the topology of $\mathcal{T}^{\Z^d}_\lambda$, each $\mathcal{V}_i$ is closed in the topology of $\mathcal{P}(\T^{\Z^d}_\lambda)$, and therefore the infinite intersection is closed.
\end{proof}
\begin{proof}[Proof of Proposition \ref{Proposition Exp Tightness Pi W}]
Since, thanks to Theorem \ref{Theorem PinB LDP 1}, $(\Pi^n_B)_{n\in\Z^+}$ satisfy an LDP relative to the topology $\tau_W$ over $\mathcal{P}(\T^{\Z^d})$, through \cite[Exercise 1.2.19]{dembo-zeitouni:97}, they must be exponentially tight relative to this topology. This means that for each $\alpha > 0$, there must exist a set $\mathcal{K} \subset \mathcal{P}(\T^{\Z^d})$, compact in the topology $\tau_W$, such that
\begin{equation}\label{eq exp tight compact 3}
\lsup{n}\frac{1}{|V_n|}\log\Pi^n(\mathcal{K}^c) = \lsup{n}\frac{1}{|V_n|}\log \Prob( \hat{\mu}^n(B) \notin \mathcal{K} )  \leq - \alpha.
\end{equation}
Let $(\delta_i,\gamma_i)_{i\in \Z^+}$ be two sequences such that $\delta_i \in (0,1/2)$, $\delta_i \to 0$ and $\gamma_i \to \infty$. Assuming for the moment that $\mathcal{K} \cap \mathcal{C}(\delta,\gamma)$ is nonempty (recall that $\mathcal{C}(\delta,\gamma)$ is defined in \eqref{defn c delta gamma}), define $\tilde{\mathcal{K}}\subset\mathcal{P}(\T^{\Z^d}_\lambda)$ to be the compact set given in Lemma \ref{Lemma Compact B K} such that $\mathcal{K} \cap \mathcal{C}(\delta,\gamma) \subseteq \tilde{\mathcal{K}} \subseteq  \mathcal{C}(\delta,\gamma)$. Our aim is to show that
\begin{equation}\label{eq : exp tightness to show}
\lsup{n}\frac{1}{|V_n|}\log\Prob( \hat{\mu}^n(B) \notin \tilde{\mathcal{K}} ) \leq -\alpha.
\end{equation}
Observe that
\begin{equation}\label{eq bound two terms}
\Prob( \hat{\mu}^n(B) \notin \tilde{\mathcal{K}} ) \leq \Prob( \hat{\mu}^n(B) \notin \mathcal{K} ) + \Prob(\hat{\mu}^n(B) \notin \mathcal{C}(\delta,\gamma)). 
\end{equation}
Now $\Prob( \hat{\mu}^n(B) \notin \mathcal{K} )\to 0$ as $n\to \infty$ by \eqref{eq exp tight compact 3}, and we will see that  $\Prob(\hat{\mu}^n(B) \notin \mathcal{C}(\delta,\gamma)) \to 0$ as $n\to \infty$. This means that $\mathcal{K} \cap \mathcal{C}(\delta,\gamma)$ must be nonempty: which justifies our previous assumption of this. 

We now find a more precise bound on $\Prob(\hat{\mu}^n(B) \notin \mathcal{C}(\delta,\gamma))$. We claim that
\begin{equation}\label{eqn exp tight claim 1}
\Prob\big(\hat{\mu}^n(B) \notin \mathcal{C}(\delta,\gamma)\big) \leq \sum_{i=1}^\infty \Prob\bigg( \sum_{j\in V_n}\norm{B^j}_T^2 > |V_n|\delta_i\gamma_i\bigg).
\end{equation}
As previously, write $\tilde{B} \in \T^{\Z^d}_\lambda$ to be the $V_n$-periodic interpolant of $(B^j)_{j\in V_n}$, i.e. such that $\tilde{B}^k := B^{k\modd V_n}$ for all $k\in \Z^d$. Now if $\hat{\mu}^n(B) \notin \mathcal{C}(\delta,\gamma)$, then there must exist $i,m\in \Z^+$ such that $\frac{1}{|V_n|} \left\lbrace\# j \in V_n \text{ s.t }\norm{S^j \tilde{B}}^2_{T,\beta_m} >\gamma_i\right\rbrace >\delta_i$. This implies that 
\begin{equation}\label{eq exp tight temp 2}
\sum_{j\in V_n} \norm{S^j \tilde{B}}^2_{T,\beta_m}> |V_n| \delta_i \gamma_i.
\end{equation}
Now, for any $m\in\Z^+$,
\begin{align*}
\sum_{l\in V_n} \norm{S^l \tilde{B}}^2_{T,\beta_m} &= \sum_{j\in V_n}\sum_{k\in\Z^d}\beta^k_m \norm{B^{(j+k)\modd V_n}}^2_T.
 \end{align*}
 Now for any $l\in V_n$ and $k\in\Z^d$, there is a unique $j\in V_n$ such that $(j+k)\modd V_n = l$. What this means is that
 \begin{equation*}
  \sum_{j\in V_n,k\in\Z^d}\beta^k_m \norm{B^{(j+k)\modd V_n}}^2_T = \sum_{l\in V_n,k\in\Z^d}\beta^k_m\norm{B^l}^2_T = \sum_{l\in V_n}\norm{B^l}_T^2.
 \end{equation*}
 Observe that the last equation is independent of $m$. We may thus infer that
\begin{align}
\Prob\left(\hat{\mu}^n(B) \notin \mathcal{C}(\delta,\gamma)\right) &\leq \sum_{i=1}^\infty \Prob\left(\sum_{j\in V_n} \norm{S^j \tilde{B}}^2_{T,\beta_m} > |V_n| \delta_i \gamma_i\right)\nonumber \\
&= \sum_{i=1}^\infty \Prob\bigg(\sum_{j\in V_n} \norm{B^j}^2_T > |V_n| \delta_i\gamma_i\bigg).\label{eq: temp reference 12}
\end{align}
We have thus established our claim \eqref{eqn exp tight claim 1}.

Let $\gamma_p \delta_p = \nu_p$. Then, letting $a= \frac{1}{4T}$, by  Chebyshev's Inequality,
\begin{align*}
\Prob\bigg(\sum_{j\in V_n}\norm{B^j}_T^2 >  |V_n|\nu_p \bigg) &= \Prob\bigg(a\sum_{j\in V_n}\norm{B^j}_T^2 >a |V_n|\nu_p \bigg) \\
&\leq \exp(-a |V_n|\nu_p)\Exp\left[\exp\big(a\sum_{j\in V_n}\norm{B^j}_T^2\big)\right] \\
&\leq \exp\big(a|V_n|\big(\grave{C} - \nu_p\big)\big),
\end{align*}
thanks to Lemma \ref{Lemma bound B}.
We may assume that $\delta_p$ and $\gamma_p$ are such that $\nu_p =\grave{C}+\frac{p\alpha}{a}$. We then find that for all $p\geq 1$,
\begin{equation*}
\Prob\bigg(\sum_{j\in V_n} \norm{B^j}^2_T > |V_n| \delta_{p}\gamma_{p}\bigg) \leq \exp\big(-p\alpha |V_n|\big).
\end{equation*}
Hence using the formula for the summation of a geometric sequence, 
\begin{align}
 \sum_{p=1}^\infty \Prob\bigg(\sum_{j\in V_n} \norm{B^j}^2_T > |V_n| \delta_p\gamma_p\bigg) &\leq \sum_{p=1}^\infty \exp\big(-p\alpha |V_n|\big)\nonumber\\
 &\leq \bigg( 1- \exp(-\alpha |V_n|)\bigg)^{-1}\exp(-\alpha |V_n|)\nonumber \\
 &\leq 2\exp(-\alpha |V_n|),\label{eq bound two terms 2}
\end{align}
for large enough $n$. This means that, through \eqref{eq exp tight compact 3}, \eqref{eq bound two terms} and \eqref{eq bound two terms 2},
\begin{align*}
&\lsup{n}\frac{1}{|V_n|}\log\Prob( \hat{\mu}^n(B) \notin \tilde{\mathcal{K}} )\\ &\leq 
 \lsup{n}\frac{1}{|V_n|}\log\bigg(\Pi^n(\mathcal{K}^c) + \Pi^n\big(\hat{\mu}^n(B) \notin \mathcal{C}(\delta,\gamma)\big)\bigg) \\
 &\leq \lsup{n}\frac{1}{|V_n|}\log\left(3\exp(-|V_n|\alpha)\right)= -\alpha.
\end{align*}
We have thus proved \eqref{eq exp tight compact 2}, as required.
\end{proof}

\section{An Application: A Fitzhugh-Nagumo Neural Network with Chemical Synapses, subject to Correlated Noise}\label{Section Example}

In this section we outline an example of a model satisfying \eqref{eq:fundamentalmult} and the assumptions of Section \ref{Section Assumptions}. We take the internal dynamics to be that of the Fitzhugh-Nagumo model, the interaction terms to be that of chemical synapses with the maximal conductances evolving according to a learning rule, and the noise to be the correlated martingale of the previous section. We take $d \in \lbrace 1,2,3\rbrace$. For $j\in V_n$, we have
\begin{align}
&dv^j_t =  dW^{n,j}_t\label{eq: fitzhugh 1}\\ &+\bigg(v^j_t - \frac{1}{3}(v^j_t)^3 - w^j_t + \sum_{k\in V_n}J^k_t(v^j,v^{(j+k)\modd V_n})f_1(v^j_t) f_2(v^{(j+k)\modd V_n}_t)\bigg)dt , \nonumber \\
&dw^j_t = \big(v^j_t + \mathfrak{a} - \mathfrak{c}w^j_t\big)dt.\label{eq: fitzhugh 2}
\end{align}
We take $w^j_0 = v^j_0 = 0$ as initial conditions. Here $\mathfrak{a}$ and $\mathfrak{c}$ are positive constants. The internal dynamics of the above equation is that of the famous Fitzhugh-Nagumo model \cite{fitzhugh:55,nagumo-arimoto-etal:62,fitzhugh:66,fitzhugh:69}.  This model distils the essential mathematical features of the Hodgkin-Huxley model, yielding excitation and transmission properties from the analysis of the biophysics of sodium and potassium flows. The variable $v$ is the `fast' variable which corresponds approximately to the voltage, and $w$ is the `slow' recovery variable which is dominant after the generation of an action potential.

We may reduce this to a one-dimensional equation by noticing that the solution of \eqref{eq: fitzhugh 2} is
\[
w^j_t = \mathfrak{c}^{-1}\int_0^t\exp\big(-\mathfrak{c}(t-s)\big) \big(v^j_s + \mathfrak{a}\big)ds.
\]
Hence we identify $U^j_t := v^j_t$ and
\begin{equation}
\mathfrak{b}_t(U^j) := U^j_t - \frac{1}{3}(U^j_t)^3 +\mathfrak{c}^{-1}\int_0^t\exp\big(-\mathfrak{c}(t-s)\big) \big(U^j_s + \mathfrak{a}\big)ds.
\end{equation}

The interaction term is a simplification of the chemical synapse models in \cite{destexhe-mainen-etal:94,ermentrout-terman:10}. It can be seen that the interaction has been decomposed into the multiplication of three terms. The terms $(J^k(\cdot,\cdot))_{k\in V_n}$ represent the maximal conductances, which are assumed to evolve according to the learning rule outlined in the following section. $J^k_s(\cdot,\cdot)$ is taken to be globally Lipschitz and bounded by $\bar{J}^k$ (which satisfies \eqref{eq Jk convergence}). The functions $f_1,f_2$ are the response functions and corresponds to the fraction of open channels: they are taken to be bounded by $\bar{f}$ and globally Lipschitz. 

The noise $W^n := (W^{n,j}_t)_{j\in V_n,t\in [0,T]}$ is taken to be the same as in Section \ref{Section Noise LDP}. In Theorem \ref{Theorem PinW} we proved a Large Deviation Principle for the sequence of laws $(\Pi^n_W)_{n\in\Z^+}$ of $\hat{\mu}^n\big(W^n\big)$. It is easy to check that the rest of the assumptions of Section \ref{Section Assumptions} are satisfied. 

\subsection{Learning Model of Synaptic Connections} \label{Section Learning Model}

One of the strengths of this paper is that the synaptic connections may evolve in time according to a learning rule. This is one way in which our work is different from mean-field models. An example of a possible model is the following classical Hebbian Learning model (refer to \cite{gerstner-kistler:02b} for a more detailed description, and in particular Equation 10.6). 

As stipulated previously, suppose that the maximal connection strength between neurons $j$ and $(j+k)$ is given by $\bar{J}^k \geq 0$. This is assumed to satisfy the condition\begin{equation}\label{eq Jk convergence}
\sum_{k\in \Z^d}\bar{J}^k < \infty.
\end{equation}

We assume that the `activity' of  neuron $j$ at time $t$ is given as $\mathfrak{v}(U^j_t)$. Here $\mathfrak{v}: \R^{2} \to \R$ is Lipschitz continuous, positive and bounded. The evolution equation is defined to be
\begin{equation}
\frac{d}{dt}J^k_t(U^{j},U^{j+k}) = J^{corr}\left(\bar{J}^k - J^k_t(U^j,U^{j+k})\right)\mathfrak{v}(U^j_t)\mathfrak{v}(U^{j+k}_t) - J^{dec}J^k_t(U^j,U^{j+k}).\label{eq:Lambdadefn}
\end{equation}
Here $J^{corr},J^{dec}$ are non-negative constants (if we let them be zero then we obtain weights which are constant in time). Initially, we stipulate that
\begin{equation}
J^k_0(U^j,U^{j+k}) := J^k_{ini}
\end{equation}
where $J^k_{ini} \in [0,\bar{J}^k]$ are constants stipulating the initial strength of the weights. It is straightforward to show that there is a unique solution to the above differential equation for all $U^j,U^{j+k} \in \C([0,T],\R^2)$. One may show that $J^k_t \leq \bar{J}^k$. In effect, the solution defines $J^k_t$ as a function $J^k_t: \C([0,t],\R^2) \times \C([0,t],\R^2) \to \R$, which can be shown to be uniformly Lipschitz in both of its variables, where $\C([0,t],\R^2)$ is endowed with the supremum norm.

Other nonlocal learning rules are possible: for a neuroscientific motivation see for example \cite{oja:82,miller-mackay:96,galtier:11}. In brief, one may assume that the synaptic connection from presynaptic neuron $k$ to postsnaptic neuron $j$ is a function of $\lbrace U^l\rbrace_{l-j\in V_m \text{ or } l-k\in V_m}$, for some fixed $m > 0$. We must then redefine the state variable at index point $j\in\Z^d$ to be the states of all the neurons in the cube centred at $j$ and of side length $(2m+1)$. One would then have to generalise the result of this paper to  having a multidimensional state vector (which would be straightforward).

\bibliographystyle{siam}
\bibliography{odyssee}
\end{document}